\documentclass{article}
\usepackage[scale=0.7,a4paper]{geometry}

\usepackage{braket,amsfonts}
\usepackage{bm}
\usepackage{amssymb} 
\usepackage{stmaryrd} 

\usepackage{mathtools} 
\usepackage[colorlinks,hypertexnames=false,citecolor=red,linkcolor=blue]{hyperref}

\usepackage{pgfplots}
\pgfplotsset{compat=1.16}

\usepackage{amsthm}
\newtheorem{theorem}{Theorem}[section]
\newtheorem{lemma}[theorem]{Lemma}
\newtheorem{remark}[theorem]{Remark}

\usepackage{algorithmic}

\usepackage{graphicx,epstopdf}

    \title{Pressure-robust staggered DG methods for the Navier-Stokes equations on general meshes\thanks{Submitted to the editors Month day, year.
    }}
    \author{
        Dohyun Kim\thanks{School of Mathematics and Computing (Computational Science and Engineering), Yonsei University, Seoul 03722, Korea
            ({\tt{kim92n@yonsei.ac.kr}}).
        }
    \and Lina Zhao\thanks{Department of Mathematics, City University of Hong Kong, Kowloon Tong, Hong Kong SAR, China
            ({\tt{linazha@cityu.edu.hk}}).
        }
    \and Eric Chung\thanks{Department of Mathematics, The Chinese University of Hong Kong, Hong Kong SAR, China
            ({\tt{tschung@math.cuhk.edu.hk}})
        }
    \and Eun-Jae Park\thanks{School of Mathematics and Computing (Computational Science and Engineering), Yonsei University, Seoul 03722, Korea
            ({\tt{ejpark@yonsei.ac.kr}}).
        }
    }

\newcommand{\dX}{\,\textup{d}\mathbf{x}}
\newcommand{\ds}{\,\textup{d}s}
\newcommand{\nn}{\mathbf{n}}
\newcommand{\tn}{\mathbf{t}}
\newcommand{\einner}[1]{\left\langle #1\right\rangle}
\newcommand{\infsup}[2]{\adjustlimits\inf_{#1} \sup_{#2}}
\newcommand{\avg}[1]{\ensuremath{\left\{\!\!\left\{ #1\right\}\!\!\right\}}}
\newcommand{\jump}[1]{\ensuremath{\left\llbracket #1\right\rrbracket}}
\newcommand{\norm}[1]{\left\|#1\right\|}
\newcommand{\tdiv}{\mathrm{div}\,}
\usepackage{xcolor}

\begin{document}

\date{}
\maketitle

\begin{abstract}
    In this paper, we design and analyze staggered discontinuous Galerkin methods of arbitrary polynomial orders for the stationary Navier-Stokes equations on polygonal meshes.
    The exact divergence-free condition for the velocity is satisfied without any postprocessing.
    The resulting method is pressure-robust so that the pressure approximation does not influence the velocity approximation.
    A new nonlinear convective term that earning non-negativity is proposed. The optimal convergence estimates for all the variables in $L^2$ norm are proved.
    Also, assuming that the rotational part of the forcing term is small enough, we are able to prove that the velocity error is independent of the Reynolds number and of the pressure.
    Furthermore, superconvergence can be achieved for velocity under a suitable projection.
    Numerical experiments are provided to validate the theoretical findings and demonstrate the performances of the proposed method.
\end{abstract}




\pagestyle{myheadings} \thispagestyle{plain} \markboth{Kim, Zhao,
Chung, and Park} {Pressure-robust SDG for NSE on general meshes}

\section{Introduction}\label{sec:intro}

The Navier-Stokes equations play an important role in fluid dynamics.
A number of finite element methods were proposed to solve the incompressible Navier-Stokes equations and accomplished many advances in the past half-century \cite{Girault1986,taylor1973,Franca1992,Cockburn05,Cockburn2009NS,Riviere14,Cesmelioglu2017,Sander18}.
Traditional inf-sup stable finite element methods do not give robust velocity approximation when large irrotational force is considered in general.
Such methods yield velocity approximations which depend on a pressure-dependent error contribution $\nu^{-1}\inf_{q\in Q_h}\|p-q_h\|_{L^2(\Omega)}$ where $\nu$ is the viscosity and $Q_h$ is the discrete pressure space.
The influence of the pressure-dependent error contribution is most pronouncing in the no-flow example which was first considered in \cite{Dorok1994}.
One of the main reason for the pressure-dependent error contribution is that the velocity approximation does not satisfy the incompressibility condition.
To satisfy the incompressibility condition, divergence-free finite element methods were proposed, see, for example \cite{Cockburn2007,Wang07,Guzman13,Guzman14,Fu2018,Wang2009}.
It is by no means trivial to construct finite element spaces that satisfy inf-sup condition and at the same time yield divergence free velocity. One of the approach exploited is to
enrich the velocity space without violating inf-sup condition. However, the construction of a suitable space to enrich velocity space is tricky. As an alternative, one can use divergence free velocity reconstruction operator to modify the right hand side \cite{Brennecke15,Linke16,Frerichs2020,Quiroz2020}.
When the nonlinear Navier-Stokes equations are considered, the reconstruction operator should be applied to the velocity for each nonlinear iteration and this leads to additional computational cost.
For the case of evolutionary incompressible Navier-Stokes equations, several pressure-robust numerical approaches were considered such as $H(\tdiv)$-conforming discontinuous Galerkin (DG) \cite{Schroeder2018, Han2020}, $H^1$-conforming mixed finite element method with grad-div stabilization \cite{Frutos2018, Arndt2014}, and continuous interior penalty methods \cite{Burman2007}.
While some of numerical simulations show optimal convergence behavior, error estimates provided in the aforementioned works are suboptimal.


In recent years, a large effort has been devoted to the design and
analysis of discretization schemes that apply to general polygonal
meshes. Among all the methods, we mention several polygonal
methods that have been designed for arbitrary polygonal orders
such as polygonal DG \cite{Cangiani2017} methods, virtual element
methods (VEM) \cite{veiga2013}, hybrid high-order (HHO)
\cite{pietro2015} methods, and weak Galerkin (WG) methods
\cite{Wang14}. In the present work, we will devise a new approach
within the framework of staggered DG methods. Staggered DG methods
as a new generation of numerical schemes were firstly introduced
by Chung and Engquist for wave propagation on triangular meshes
\cite{chung2006}. Since then it has been successfully applied for
various problems
\cite{chung2009,ChungLee11,chung2013,Kim13,chung2013b,cheung2015,ChungDu17,ChungQiu17}.
Recently, Zhao and Park \cite{zhao2018} extended this method to
general polygonal meshes for the Poisson equation. Then, a high
order staggered DG method for general second-order elliptic
problems is developed in \cite{Zhao2020e}, and it is applied to
various physical problems arising from practical applications
\cite{Zhao2019,Zhao2020e,dohyun2020,Zhao2020d,Zhao2020c,Zhao2020b,Zhao2020a,Zhao2021}.
For a pressure-robust method, Zhao et al. \cite{Zhao2020arxiv}
proposed a lowest-order pressure-robust staggered DG method for
the Stokes equations. They used a reconstruction operator based on
an $H$(div) conforming function space on polygonal meshes.
However, its extension to high-order methods is not trivial.

The objective of this paper is to design and analyze a pressure
robust staggered DG method of arbitrary polynomial orders on
polygonal meshes for the Navier-Stokes problem. The discrete
formulation involves velocity gradient, velocity and pressure, and
the continuities for all the variables are staggered on the
inter-element boundaries in line with \cite{Zhao2020c}. The
staggered continuity naturally gives an inter-element flux term
which is free of stabilization parameters and meanwhile ensures
the inf-sup stability of the resulting bilinear forms. The
proposed method is pressure-robust without enriching the velocity
space nor introducing the divergence-preserving reconstruction
operator. To the best of the authors' knowledge, this is the first
pressure-robust method on polygonal meshes without enriching the
discrete spaces and without establishing reconstruction operators.
As such, the construction of the method is relatively simple by
using standard polynomial spaces and the extension to arbitrary
polynomial orders is straightforward. Another novel contribution
lies in the design of a new nonlinear convective term that
guarantees non-negativity. We prove that the resulting method
yields a divergence-free velocity approximation without any
postprocessing, which is particularly important for Navier-Stokes
equations. In addition, the unique solvability of the discrete
solution is proved under a smallness condition involving only the
solenoidal part of the body force. Importantly, the convergence
estimate for velocity is proved to be independent of the pressure
and of the viscosity. The staggered DG method designed herein
offers some unique features which makes it advantageous over other
polygonal methods: It is locally conservative over each
dual-element; superconvergence can be obtained;
it is stable without numerical flux nor penalty term; exact
divergence free condition is satisfied. It should be also
emphasized that developing a divergence free element on 3D is much
more difficult and our approach can be extended to 3D
straightforwardly. The current paper focuses on 2D to present the
core of the method and simplify the presentation.



The organization of the paper is as follows.
In Section~\ref{sec:sdg}, we introduce the pressure-robust staggered DG method on polygonal meshes and give some useful lemmas including the divergence-free property of the velocity approximation.
In Section~\ref{sec:unisolv}, we prove the existence and uniqueness of the nonlinear discrete problem using fixed-point argument.
Then \textit{a priori} error estimates for all variables are proved in Section~\ref{sec:pri}.
In Section~\ref{sec:num}, several numerical experiments are conducted to verify the pressure robustness and the optimality of the method.
We end in Section~\ref{sec:con} with some concluding remarks.

\section{Staggered discontinuous Galerkin method}\label{sec:sdg}
Let $\Omega\subset\mathbb{R}^2$ be a bounded simply connected polygonal domain with Lipschitz boundary $\partial\Omega$.
For given data $\bm{f}\in [L^2(\Omega)]^2$ and $\bm{g}\in [H^{1/2}(\partial\Omega)]^2$, the incompressible Navier-Stokes problem in the conservative form seeks the unknown velocity $\bm{u}$ and the pressure $p$ satisfying
\begin{equation}\label{eq:model-conserve}
    \begin{aligned}
        -\nu\Delta \bm{u}+\tdiv(\bm{u}\otimes\bm{u})+\nabla p&=\bm{f}&&\text{in }\Omega,\\
        \nabla\cdot\bm{u}&=0&&\text{in }\Omega,\\
        \bm{u}&=\bm{g}&&\text{on }\partial\Omega,\\
        \einner{p}_\Omega&=0.
    \end{aligned}
\end{equation}
Here, $\einner{p}_\Omega=\int_\Omega p\dX$ and $\nu>0$ is a real number representing the kinematic viscosity of the fluid.
We introduce an auxiliary variable $\bm{G}=\nu\nabla\bm{u}$, then \eqref{eq:model-conserve} can be recast into the following first order system
\begin{equation}\label{eq:model-conserve-with-G}
    \begin{aligned}
        \nu^{-1}\bm{G}-\nabla\bm{u}&=0,&&\text{in }\Omega,\\
        -\tdiv\bm{G}+\tdiv(\bm{u}\otimes\bm{u})+\nabla p&=\bm{f},&&\text{in }\Omega,\\
        \nabla\cdot\bm{u}&=0,&&\text{in }\Omega,\\
        \bm{u}&=\bm{g},&&\text{on }\partial\Omega,\\
        \einner{p}_\Omega&=0.
    \end{aligned}
\end{equation}
In the remainder of this paper, we assume $\bm{g}=\bm{0}$ for simplicity.

We denote the Sobolev space $W^{k,p}(\Omega)$ and we write $H^k(\Omega)$ when $p=2$.
Those spaces are equipped with norm $\norm{\cdot}_{W^{k,p}(\Omega)}$ and $\norm{\cdot}_{H^k(\Omega)}$.
$W^{k,p}_0(\Omega)$ and $H^k_0(\Omega)$ denote the closure of $C^\infty_c(\Omega)$ with respective norms.
Here, $C^\infty_c(\Omega)$ is the infinitely differentiable function spaces with compact support.
We also define
\begin{equation*}
    H(\tdiv;\Omega)=\{\bm{v}\in [L^2(\Omega)]^2:\nabla\cdot\bm{v}\in L^2(\Omega)\},\quad L^2_0(\Omega)=\{q\in L^2(\Omega):\einner{q}_\Omega=0\}.
\end{equation*}
In the sequel we use $C$ to denote a generic positive constant which may have different values at different occurrences.
The weak solution to \eqref{eq:model-conserve-with-G} is defined by
\begin{equation}
    \begin{aligned}
        (\nu^{-1}\bm{G}-\nabla\bm{u},\bm{H})&=0&&\forall \bm{H}\in [L^2(\Omega)]^{2\times 2},\\
        (\bm{G}-\bm{u}\otimes\bm{u}-p\bm{I},\nabla\bm{v})&=(\bm{f},\bm{v})&&\forall \bm{v}\in [H^1_0(\Omega)]^2,\\
        (\nabla\cdot\bm{u},q)&=0&&\forall q\in L^2_0(\Omega).
    \end{aligned}
\end{equation}



We can prove the following stability estimate proceeding similarly to \cite{Quiroz2020} and the proof is omitted for simplicity.
\begin{lemma}\label{lem:bound_by_f0}
    For $\bm{f}\in [L^2(\Omega)]^2$, let $\bm{f}=\bm{f}_0+\nabla\chi$ be the Helmholtz-Hodge decomposition with $(\bm{f}_0,\chi)\in H(\tdiv;\Omega)\times (H^1(\Omega)\cap L^2_0(\Omega))$. Further, $\bm{f}_0$ satisfies $\nabla \cdot \bm{f}_0=0$.
    If $(\bm{G},\bm{u},p)$ is the solution to \eqref{eq:model-conserve-with-G},
   then, there holds
    \begin{equation*}
        \norm{\bm{u}}_{H^1(\Omega)}\leq C\nu^{-1}\norm{\bm{f}_0}_{L^2(\Omega)}.
    \end{equation*}
    In other words, $\bm{u}$ is independent of the irrotational force $\nabla\chi$.
\end{lemma}

We now introduce some basic notations regarding staggered meshes that will be exploited in the construction of the staggered DG method.
Let $\mathcal{S}$ be a star-shaped polygonal partition of $\Omega$, and the set of all the edges is called primal-edges and is denoted by $\mathcal{F}_{pr}$.
In addition, we use $\mathcal{F}_{pr}^i$ and $\mathcal{F}_{pr}^b$ to stand for the interior edges and boundary edges, respectively.
For each primal-element $T\in \mathcal{S}$, we select an interior point $\bm{x}$ and connect it to all the vertices of $T$, thereby sub-triangular grids are generated during this process.
The resulting triangulation is denoted as $\mathcal{T}_h$ and the edges generated during the subdivision process are called dual-edges and are denoted by $\mathcal{F}_{dl}$.
We rename the union of sub-triangles sharing the common vertex $\bm{x}$ as $S(\bm{x})$.
Here, we choose $\bm{x}$ to be an interior point of the kernel of $S(\bm{x})\in \mathcal{S}$ and we use $\mathcal{N}$ to represent the union of all interior points $\bm{x}$.
For each interior edge $e\in \mathcal{F}_{pr}^i$, we use $\mathcal{D}(e)$ to stand for the dual-mesh, which is the union of two triangles in $\mathcal{T}_h$ sharing the common edge $e$.
For each boundary edge $e\in \mathcal{F}_{pr}^b$, we use $\mathcal{D}(e)$ to denote a triangle in $\mathcal{T}_h$ having the edge $e$, see Figure~\ref{fig:mesh_exam} for an illustration.
For each edge $e\in \mathcal{F}_h=\mathcal{F}_{pr}\cup \mathcal{F}_{dl}$, we define a unique normal vector $\bm{n}$ and tangential vector $\bm{t}$ by the outward normal vector and counter-clockwise tangential vector of one of its neighboring element, respectively.
\begin{figure}\label{fig:mesh_exam}
    \centering
    \includegraphics[width=0.8\textwidth]{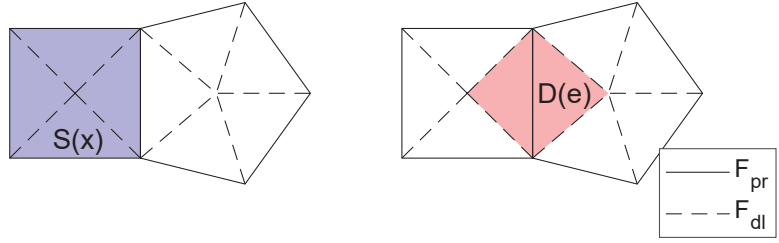}
    \caption{
        Schematic of primal- and dual-meshes.
        Solid lines are primal-edges $\mathcal{F}_{pr}$ and dashed lines are dual-edges $\mathcal{F}_{dl}$.
        A polygon surrounded by primal-edges is called primal-element $S(\bm{x})$ and quadrilaterals surrounded by dual-edges are called dual-element $\mathcal{D}(e)$.}
\end{figure}

For later analysis, we employ the general mesh regularity assumption (cf. \cite{Cangiani17,zhao2018}):
For every element $S(\bm{x})\in\mathcal{S}$ and every edge $e\in\partial S(\bm{x})$, it satisfies $h_e\geq \rho_E h_{S(\bm{x})}$ for a positive constant $\rho_E$, where $h_e$ is the length of edge $e$ and $h_{S(\bm{x})}$ is the diameter of $S(\bm{x})$.
Second, each element $S(\bm{x})$ in $\mathcal{S}$ is star-shaped with respect to a ball of radius $\geq \rho_B h_{S(\bm{x})}$, where $\rho_B$ is a positive constant.
We remark that the above assumptions ensure that the triangulation $\mathcal{T}_h$ is shape regular.
We employ the general mesh regularity assumption just for the sake of simplicity. In fact, our numerical results indicate that the proposed method allows elements with arbitrarily small edges and a rigorous analysis for a mesh with small edges will be present in our future work.

The jump $\jump{\cdot}$ and the average $\avg{\cdot}$ is defined by $\jump{v}_e=v|_{\tau_+}-v|_{\tau_-}$ and $\avg{v}_e=(v|_{\tau_+}+v|_{\tau_-})/2$ where $\tau_+$ and $\tau_-$ are two elements in $\mathcal{T}_h$ sharing the common edge $e\in \mathcal{F}_{pr}^i\cup \mathcal{F}_{dl}$.
For $e\in \mathcal{F}_{pr}^b$, we simply take $\jump{v}_e=\avg{v}_e=v|_{\tau_+}$.
The subscript $e$ will be omitted when there is no ambiguity.
We denote the $L^2$-inner product by $(f,g)_\tau=\int_\tau fg\dX$ for 2D and $\einner{f,g}_e=\int_efg\ds$ for 1D.
When $\bm{f}$ and $\bm{g}$ are vectors (or tensors), then $(\cdot,\cdot)_{\tau}$ and $\einner{\cdot,\cdot}_e$ are defined by their component-wise sum. That is, when $\bm{f},\bm{g}:\mathbb{R}^2\rightarrow\mathbb{R}^2$,
\begin{equation*}
    (\bm{f},\bm{g})_\tau=\int_\tau\bm{f}\cdot\bm{g}\dX,\quad\einner{\bm{f},\bm{g}}_e=\int_e\bm{f}\cdot\bm{g}\ds,
\end{equation*}
and when $\bm{f},\bm{g}:\mathbb{R}^2\rightarrow \mathbb{R}^{2\times 2}$,
\begin{equation*}
    (\bm{f},\bm{g})_\tau=\int_\tau \bm{f}:\bm{g}\dX,\quad\einner{\bm{f},\bm{g}}_e=\int_e \bm{f}:\bm{g}\ds.
\end{equation*}
Here, $A:B=\sum_{ij}A_{ij}B_{ij}$ is the Frobenius inner product.
The discrete spaces are defined based on the definition described above.
Let $(H_h,V_h,Q_h)$ be the discrete spaces defined by
\begin{align*}
    H_h&=\{{\bm H}\in [\mathbb{P}_k(\mathcal{T}_h)]^{2\times 2}:
        \jump{\bm{H}\nn}_e=\bm{0}\;\forall e\in\mathcal{F}_{pr}^0,\;\jump{\tn\cdot \bm H\nn}_e=0\;\forall e\in\mathcal{F}_{dl}\},\\
    V_h&=\{\bm{v}\in[\mathbb{P}_k(\mathcal{T}_h)]^2:
        \jump{\bm{v}\cdot\nn}_e=0\;\forall e\in\mathcal{F}_{dl}\},\\
    Q_h&=\{q\in \mathbb{P}_k(\mathcal{T}_h):\jump{q}_e=0\;\forall e\in\mathcal{F}_{pr}^0\},
\end{align*}
where $\mathbb{P}_k(\mathcal{T}_h)$ is a complete polynomial space of degree less than or equal to $k$ on each triangle $\tau\in\mathcal{T}_h$.
The discrete space $V_h$ is equipped with the norm
\begin{equation*}
    \norm{\bm{v}}_{h}^2=\norm{\nabla\bm{v}}_{L^2(\mathcal{T}_h)}^2 + \sum_{e\in\mathcal{F}_{pr}}h_e^{-1}\norm{\jump{\bm{v}}}_{L^2(e)}^2+\sum_{e\in\mathcal{F}_{dl}}h_e^{-1}\norm{\jump{\bm{v}\cdot\tn}}_{L^2(e)}^2.
\end{equation*}
Here, $\norm{\cdot}_{L^2(\mathcal{T}_h)}$ is the discrete $L^2$-norm on the triangulation $\mathcal{T}_h$.
We also define the discrete $L^4$-norm for $V_h$ by
\begin{equation}\label{eq:discL4norm}
    \norm{\bm{v}}_{0,4,h}^4=\norm{\bm{v}}_{L^4(\mathcal{T}_h)}^4
    +\sum_{e\in\mathcal{F}_h}h_e^{-1}\norm{\avg{\bm{v}}}_{L^4(e)}^4
\end{equation}
and discrete $H^1$-seminorm for $Q_h$ by
\begin{equation*}
    \norm{q}_{1,h}^2=\norm{\nabla q}_{L^2(\mathcal{T}_h)}^2+\sum_{e\in\mathcal{F}_{dl}}h_e^{-1}\norm{\jump{q}}_{L^2(e)}^2.
\end{equation*}
To impose the mean zero condition for the pressure variable, we introduce
\begin{equation*}
    Q_h^0=\{q\in Q_h:\einner{q}_\Omega=0\}.
\end{equation*}
Note that $\norm{\cdot}_{1,h}$ is a norm on $Q_h^0$.

In the following, we introduce the (discrete) trace inequality and the discrete Sobolev embedding theorem.

\begin{lemma}\label{lem:discIneqs}
    There exists $C$ independent of $h$ such that for all $e\in\mathcal{F}_h$ and $e\subset \partial \tau\in \mathcal{T}_h$ (cf. \cite{Karakashian98})
    \begin{align}
          \norm{\bm{v}}_{L^4(e)}\leq C(h_\tau^{-1/4}\norm{\bm{v}}_{L^4(\tau)}+h_\tau^{1/4}\norm{\nabla \bm{v}}_{L^2(\tau)})\quad\forall\bm{v}\in H^1(\tau).\label{eq:traceL4}
    \end{align}
    Also, we have (cf. \cite{Pietro2012})
    \begin{equation*}
        \norm{\bm{v}}_{L^2(e)}\leq Ch_\tau^{-1/2}\norm{\bm{v}}_{L^2(\tau)}\quad\forall\bm{v}\in V_h.\label{eq:discTrace}
    \end{equation*}

\end{lemma}

\begin{lemma}{\cite{pietro2010}}\label{lem:discSobolev}
    There exists $C$ independent of $h$ such that
    \begin{equation*}
        \norm{\bm{v}}_{0,4,h}\leq C\norm{\bm{v}}_h\quad\forall \bm{v}\in V_h
    \end{equation*}
    and
    \begin{equation}
        \norm{\bm{v}}_{L^q(\Omega)}\leq C\norm{\bm{v}}_h\quad 1\leq q\leq 6,\;\forall\bm{v}\in V_h.\label{eq:Lq}
    \end{equation}
\end{lemma}

We can define the following degrees of freedom in the spirit of \cite{chung2009,Zhao2020c}.
\begin{lemma}[Degrees of freedom]\label{lem:dofs}
    Any function $\bm{H}\in H_h$ is uniquely determined by the following degrees of freedom: 
    \begin{itemize}
        \item[(XD1)]
            For $e\in \mathcal{F}_{pr}$, we have
            \begin{equation*}
                \Phi^{pr}_e(\bm{H}) := \einner{\bm{H}\nn,\bm{p}_k}_e\quad\forall\bm{p}_k\in[\mathbb{P}_k(e)]^2.
            \end{equation*}
        \item[(XD2)]
            For $e\in \mathcal{F}_{dl}$, we have
            \begin{equation*}
                \Phi^{dl}_e(\bm{H}) := \einner{\tn\cdot\bm{H}\nn, p_k}_e\quad\forall p_k\in\mathbb{P}_k(e).
            \end{equation*}
        \item[(XD3)]
            For each $\tau\in \mathcal{T}_h$, we can obtain
            \begin{equation*}
                \Phi_\tau(\bm{H}) := (\bm{H},\bm{p}_{k-1})_\tau\quad \forall \bm{p}_{k-1}\in [\mathbb{P}_{k-1}(\tau)]^{2\times 2}.
            \end{equation*}
    \end{itemize}
    Similarly, any function $\bm{v}\in V_h$ is uniquely determined by the following degrees of freedom:
    \begin{itemize}
        \item[(VD1)]
            For $e\in \mathcal{F}_{dl}$, we have
            \begin{equation*}
                \bm{\phi}_e(\bm{v}) := \einner{\bm{v}\cdot\nn,p_k}_e\quad \forall p_k\in \mathbb{P}_k(e).
            \end{equation*}
        \item[(VD2)]
            For each $\tau\in \mathcal{T}_h$, we can obtain
            \begin{equation*}
                \bm{\phi}_\tau(\bm{v}) := (\bm{v},\bm{p}_{k-1})_\tau\quad \forall \bm{p}_{k-1}\in [\mathbb{P}_{k-1}(\tau)]^2.
            \end{equation*}
    \end{itemize}
    Finally, any function $q\in Q_h$ is uniquely determined by
    \begin{itemize}
        \item[(SD1)]
            For $e\in \mathcal{F}_{pr}$, we have
            \begin{equation*}
                \phi_e(q) := \einner{ q,p_k}_e\quad \forall p_k\in \mathbb{P}_k(e).
            \end{equation*}
        \item[(SD2)]
            For each $\tau\in \mathcal{T}_h$, we can obtain
            \begin{equation*}
                \phi_\tau(q) := (q,p_{k-1})_\tau\quad \forall p_{k-1}\in \mathbb{P}_{k-1}(\tau).
            \end{equation*}
    \end{itemize}
\end{lemma}
\noindent
Using the degrees of freedom defined in Lemma~\ref{lem:dofs}, we define the interpolation operators so that
\begin{equation*}
    \bm{\Phi}(\bm{H}-\Pi_h\bm{H})=0,\quad\bm{\phi}(\bm{v}-J_h\bm{v})=0,\quad\phi(q-I_hq)=0.
\end{equation*}
By the polynomial preserving property of the interpolation operators, we obtain the following lemma with minor modification of \cite{chung2009}.
\begin{lemma}\label{lem:intperr}
    There exists $C$ independent of $h$ such that
    \begin{equation*}
        \begin{aligned}
        \norm{\bm{H}-\Pi_h\bm{H}}_{L^2(\Omega)}&\leq Ch^{k+1}\norm{\bm{H}}_{H^{k+1}(\Omega)}&&\forall\bm{H}\in [H^{k+1}(\Omega)]^{2\times 2},\\
        \norm{\bm{v}-J_h\bm{v}}_{L^2(\Omega)}&\leq Ch^{k+1}\norm{\bm{v}}_{H^{k+1}(\Omega)}&&\forall\bm{v}\in [H^{k+1}(\Omega)]^2,\\
        \norm{\bm{v}-J_h\bm{v}}_{0,4,h}&\leq Ch^{k+1}\norm{\bm{v}}_{W^{k+1,4}(\Omega)}&&\forall\bm{v}\in [W^{k+1,4}(\Omega)]^2,\\
        \norm{q-I_hq}_{L^2(\Omega)}&\leq Ch^{k+1}\norm{q}_{H^{k+1}(\Omega)}&&\forall q\in H^{k+1}(\Omega).
        \end{aligned}
    \end{equation*}
\end{lemma}

Then the staggered discontinuous Galerkin method for \eqref{eq:model-conserve} is defined as follows: Find $(\bm{G}_h,\bm{u}_h,p_h)\in H_h\times V_h\times Q_h^0$ such that
\begin{subequations}\label{eq:discform}
    \begin{align}
        \label{eq:discform1}
        \nu^{-1}(\bm{G}_h,\bm{H}) + B_h^*(\bm{u}_h,\bm{H})&=\sum_{e\in\mathcal{F}_{pr}^b}\einner{\bm{g},\bm{H}\nn}_e&&\forall \bm H\in H_h,\\
        \label{eq:discform2}
        B_h(\bm{G}_h,\bm{v})+N_h(\bm{u}_h;\bm{u}_h,\bm{v})+b_h^*(p_h,\bm{v})&=(\bm{f},\bm{v})
        +\sum_{e\in\mathcal{F}_h^b}\einner{|\bm{g}\cdot\nn|-\bm{g}\cdot\nn,\bm{g}\cdot\bm{v}}_e&&\forall\bm{v}\in V_h,\\
        \label{eq:discform3}
        b_h(\bm{u}_h,q)&=\sum_{e\in\mathcal{F}_{pr}^b}\einner{\bm{g}\cdot\nn, q}_e&&\forall q\in Q_h^0.
    \end{align}
\end{subequations}
Here,
\begin{align}
    \label{eq:Bh}
    B_h(\bm{H},\bm{v})
        &=\sum_{\tau\in\mathcal{T}_h}(\bm{H},\nabla\bm{v})_\tau
        -\sum_{e\in\mathcal{F}_{pr}}\einner{\bm{H}\nn,\jump{\bm{v}}}
        -\sum_{e\in\mathcal{F}_{dl}}\einner{\tn\cdot\bm{H}\nn,\jump{\bm{v}\cdot\tn}}_e,\\
    \label{eq:Bh*}
    B_h^*(\bm{v},\bm{H})
        &=\sum_{\tau\in\mathcal{T}_h}(\bm{v},\tdiv \bm{H})_\tau
        -\sum_{e\in\mathcal{F}_{dl}}\einner{\bm{v}\cdot\nn,\jump{\nn\cdot\bm{H}\nn}}_e,\\
    \label{eq:bh}
    b_h(\bm{v},q)
        &=-\sum_{\tau\in\mathcal{T}_h}(\bm{v},\nabla q)_\tau
        +\sum_{e\in\mathcal{F}_{dl}}\einner{\bm{v}\cdot\nn,\jump{q}}_e,\\
    \label{eq:bh*}
    b_h^*(q,\bm{v})
        &=-\sum_{\tau\in\mathcal{T}_h}(q, \nabla\cdot\bm{v})_\tau
        +\sum_{e\in\mathcal{F}_{pr}}\einner{q, \jump{\bm{v}\cdot\nn}}_e,\\
    \label{eq:Nh}
    N_h(\bm{w};\bm{\psi},\bm{v})
        &=-\sum_{\tau\in\mathcal{T}_h}(\bm{\psi}\otimes \bm{w}, \nabla \bm{v})_\tau
        +\sum_{e\in\mathcal{F}_h^0}\einner{\avg{\bm{w}\cdot\nn},\avg{\bm{\psi}}\cdot \jump{\bm{v}}}_e\nonumber\\
        &\quad+\sum_{e\in\mathcal{F}_h}\einner{|\avg{\bm{w}\cdot\nn}|,\jump{\bm{\psi}}\cdot \jump{\bm{v}}}_e,
\end{align}
for any $(H,\bm{v},q)\in H_h\times V_h\times Q_h^0$ and $(\bm{\psi},\bm{w})\in V_h\times V_h$.

In the remainder of this paper, we take $\bm{g}=\bm{0}$ for simplicity.
Note that the formulation is based on the conservative form of the Navier-Stokes equation and the upwind term is added to ensure the non-negativity, see Lemma~\ref{lem:Nh-nonnegativity}. 
Also, integration by parts and the definitions of the discrete spaces lead to the following discrete adjoint properties
\begin{equation}\label{eq:adjoint}
    \begin{aligned}
        B_h(\bm{H},\bm{v})&=-B_h^*(\bm{v},\bm{H})&&\forall (\bm{H},\bm{v})\in H_h\times V_h,\\
        b_h(\bm{v},q)&=-b_h^*(q,\bm{v})&&\forall (\bm{v},q)\in V_h\times Q_h.
    \end{aligned}
\end{equation}
By the definition of the discrete bilinear forms and the interpolation operators, we have the following lemma.
\begin{lemma}\label{lem:intp}
    Assume that $(\bm{G},\bm{u},p)\in ([H(\tdiv;\Omega)]^2\cap [H^{1/2+\epsilon}(\Omega)]^{2\times2})\times [H^1(\Omega)]^2\times (L^2(\Omega)\cap H^{1/2+\epsilon}(\Omega))$ with $\epsilon>0$.
    Then, the following equalities hold:
    \begin{equation*}
        \begin{aligned}
            B_h(\Pi_h\bm{G}-\bm{G},\bm{v})&=0&&\forall \bm{v}\in V_h,&b_h(J_h\bm{u}-\bm{u},q)&=0&&\forall q\in Q_h,\\
            B_h^*(J_h\bm{u}-\bm{u},\bm{H})&=0&&\forall \bm{H}\in H_h,&b_h^*(I_hp-p,\bm{v})&=0&&\forall \bm{v}\in V_h.
        \end{aligned}
    \end{equation*}
\end{lemma}
We have the discrete inf-sup condition for $B_h(\cdot,\cdot)$ and $b_h(\cdot,\cdot)$ (cf. \cite{chung2009,Zhao2020c}).
\begin{lemma}\label{lem:infsup}
    There exists $C_B$ and $C_b$ independent of $h$ such that
    \begin{equation*}
        0<C_B\leq \infsup{\bm{v}\in V_h}{\bm{H}\in H_h}\frac{B_h(\bm{H},\bm{v})}{\norm{\bm{H}}_{L^2(\Omega)}\norm{\bm{v}}_h},\quad 0<C_b\leq \infsup{q\in Q_h}{\bm{v}\in V_h}\frac{b_h^*(q,\bm{v})}{\norm{\bm{v}}_{L^2(\Omega)}\norm{q}_{1,h}}.
    \end{equation*}
\end{lemma}


Integration by parts reveals the consistency of the nonlinear trilinear form $N_h(\cdot;\cdot,\cdot)$.
\begin{lemma}[Consistency]\label{lem:Nh-consistency}
    Let $\bm{u}\in [H^1_0(\Omega)]^2$ be the solution to \eqref{eq:model-conserve}. Then the following holds:
    \begin{equation*}
        N_h(\bm{u};\bm{u},\bm{v}_h)=(\tdiv(\bm{u}\otimes\bm{u}),\bm{v}_h)\quad\forall \bm{v}_h\in V_h.
    \end{equation*}
\end{lemma}
In the next lemma, we state the non-negative property for $N_h(\cdot;\cdot,\cdot)$, which is crucial for the subsequent analysis.
\begin{lemma}[Non-negativity]\label{lem:Nh-nonnegativity}
    Let $\bm{w}\in H(\tdiv;\Omega)$ with $\nabla\cdot\bm{w}=0$. Then we have
    \begin{equation*}
        N_h(\bm{w};\bm{v},\bm{v})\geq 0\quad\forall \bm{v}\in V_h+[H^1_0(\Omega)]^2.
    \end{equation*}
\end{lemma}
\begin{proof}
    Let $\tau\in \mathcal{T}_h$ be given. From integration by parts, we have
    \begin{equation}\label{eq:Nh-intbypart}
        (\bm{v}\otimes\bm{w},\nabla \bm{v})_\tau=-(\tdiv(\bm{v}\otimes\bm{w}),\bm{v})_\tau+\einner{(\bm{v}\otimes\bm{w})\nn,\bm{v}}_{\partial\tau}.
    \end{equation}
    Since $\nabla\cdot\bm{w}=0$, integration by parts and straightforward computation yield
    \begin{equation*}
        (\tdiv(\bm{v}\otimes\bm{w}),\bm{v})_\tau=\frac{1}{2}\einner{\bm{w}\cdot\nn,\bm{v}\cdot\bm{v}}_{\partial\tau}
    \end{equation*}
    and
    \begin{equation*}
        \einner{(\bm{v}\otimes\bm{w})\nn,\bm{v}}_{\partial\tau}=\einner{\bm{w}\cdot\nn,\bm{v}\cdot\bm{v}}_{\partial\tau}.
    \end{equation*}
    Then \eqref{eq:Nh-intbypart} can be rewritten as
    \begin{equation*}
        (\bm{v}\otimes\bm{w},\nabla\bm{v})_\tau=\frac{1}{2}\einner{\bm{w}\cdot\nn,\bm{v}\cdot\bm{v}}_{\partial\tau}.
    \end{equation*}
    Summing over all the elements $\tau\in \mathcal{T}_h$, we have
    \begin{align*}
        \sum_{\tau\in\mathcal{T}_h}(\bm{v}\otimes\bm{w},\nabla\bm{v})_\tau
        &=\frac{1}{2}\sum_{\tau\in\mathcal{T}_h}\einner{\bm{w}\cdot\nn,\bm{v}\cdot\bm{v}}_{\partial\tau}\\
        &=\sum_{e\in\mathcal{F}_h^0}\einner{\bm{w}\cdot\nn,\avg{\bm{v}}\cdot\jump{\bm{v}}}_e
        + \frac{1}{2}\sum_{e\in\mathcal{F}_h^b}\einner{\bm{w}\cdot\nn,\bm{v}\cdot\bm{v}}_e.
    \end{align*}
    This and the definition of $N_h(\cdot;\cdot,\cdot)$ yield
    \begin{align*}
        N_h(\bm{w};\bm{v},\bm{v})
        &=-\frac{1}{2}\sum_{e\in\mathcal{F}_h^b}\einner{\bm{w}\cdot\nn,\bm{v}\cdot\bm{v}}_e
        +\sum_{e\in\mathcal{F}_h}\einner{|\bm{w}\cdot\nn|,\jump{\bm{v}}\cdot\jump{\bm{v}}}_e\\
        &\geq \frac{1}{2}\sum_{e\in\mathcal{F}_h}\einner{|\bm{w}\cdot\nn|,\jump{\bm{v}}\cdot\jump{\bm{v}}}_e\\
        &\geq 0.
    \end{align*}
\end{proof}

By using H\"older's inequality and the discrete Sobolev embedding theorem, we obtain the following lemma.
\begin{lemma}[Boundedness]\label{lem:Nh-bound}
    For any $\bm{z}_h,\bm{v}_h,\bm{w}_h\in V_h$, it holds
    \begin{equation*}
        N_h(\bm{z}_h;\bm{v}_h,\bm{w}_h)\leq C_N\norm{\bm{z}_h}_h\norm{\bm{v}_h}_h\norm{\bm{w}_h}_h.
    \end{equation*}
\end{lemma}

To ease later analysis, we define the divergence-free subspace of $V_h$ by
\begin{equation*}
    V_h^{\tdiv}=\{\bm{v}\in V_h:\bm{v}\in H(\tdiv;\Omega),\;\nabla\cdot\bm{v}=0\}.
\end{equation*}
This subspace plays an important role in subsequent sections.
We close this section by observing the connection between $b_h(\cdot,\cdot)$ and $V_h^\tdiv$.
\begin{lemma}\label{lem:divfree}
    If $\bm{v}_h\in V_h$ satisfies
    \begin{equation}\label{eq:divfree}
        b_h(\bm{v}_h,q)=0\quad\forall q\in Q_h^0,
    \end{equation}
    then $\bm{v}_h\in H(\tdiv;\Omega)$ and it is divergence-free.
\end{lemma}
\begin{proof}
    Note that
    \begin{equation*}
        b_h(\bm{v}_h,c)=0
    \end{equation*}
    for any constant $c$ by the definition.
    Therefore, the condition \eqref{eq:divfree} is equivalent to
    \begin{equation*}
        b_h(\bm{v}_h,q)=0\quad\forall q\in Q_h.
    \end{equation*}
    By \eqref{eq:adjoint}, we obtain
    \begin{equation*}
        0=b_h(\bm{v}_h,q)=b_h^*(q,\bm{v}_h)\quad\forall q\in Q_h.
    \end{equation*}
    Take $q$ so that (cf. (SD1)-(SD2))
    \begin{equation*}
        \begin{aligned}
            (q+\nabla\cdot\bm{v}_h,p_{k-1})_\tau&=0&&\forall p_{k-1}\in\mathbb{P}_{k-1}(\tau),\;\tau\in\mathcal{T}_h,\\
            \einner{q-\jump{\bm{v}_h\cdot\nn},p_k}_e&=0&&\forall p_k\in\mathbb{P}_k(e),\;e\in\mathcal{F}_{pr}.
        \end{aligned}
    \end{equation*}
    Then we obtain
    \begin{equation*}
        0=b_h^*(q,\bm{v}_h)=\sum_{\tau\in\mathcal{T}_h}\norm{\nabla\cdot\bm{v}_h}_{L^2(\tau)}^2+\sum_{e\in\mathcal{F}_{pr}}\norm{\jump{\bm{v}_h\cdot\nn}}_{L^2(e)}^2.
    \end{equation*}
    Therefore, $\jump{\bm{v}_h\cdot\nn}=0$ for all $e\in\mathcal{F}_{pr}$.
    Since $\jump{\bm{v}_h\cdot\nn}_e=0$ for all $e\in\mathcal{F}_{dl}$ by definition of $V_h$, we have $\bm{v}_h\in H(\tdiv;\Omega)$.
    Furthermore, since $\nabla\cdot(\bm{v}_h|_\tau)=0$ for all $\tau\in\mathcal{T}_h$, we have $\nabla\cdot\bm{v}_h=0$.
\end{proof}

\begin{remark}(divergence free velocity).
Our proposed scheme can yield a divergence-free velocity by following Lemma~\ref{lem:divfree}, which is a desirable feature. Thanks to the divergence-free property and the specially designed term for the nonlinear convective term, we are able to prove that the convergence estimates are independent of the pressure variable and the coefficient $\nu$ under a suitable assumption on the source term $\bm{f}$. Unlike the existing works on polygonal meshes \cite{chen2018,Frerichs2020,Mu2020,Zhao2020arxiv}, we do not require velocity reconstruction, which can greatly reduce the computational complexity and ease the construction of the method.
\end{remark}

\section{Existence and uniqueness}\label{sec:unisolv}
In this section, we discuss the existence and uniqueness of the solution to \eqref{eq:discform}.
A solution operator $T_h:V_h^\tdiv\rightarrow V_h$ is defined as follows: For given $\bm{z}_h\in V_h$, find $\bm{w}_h=T_h(\bm{v}_h)\in V_h$ so that
\begin{equation}\label{eq:Th}
    A_h((\bm{S}_h,\bm{w}_h,r_h),(\bm{H},\bm{v},q))+N_h(\bm{z}_h;\bm{w}_h,\bm{v})=(\bm{f},\bm{v})\quad\forall (\bm{H},\bm{v},q)\in H_h\times V_h\times Q_h
\end{equation}
for some $\bm{G}_h\in H_h$ and $p_h\in Q_h^0$.
Here,
\begin{equation*}
    A_h((\bm{S},\bm{w},r),(\bm{H},\bm{v},q))=\nu^{-1}(\bm{S},\bm{H})+B_h^*(\bm{w},\bm{H})+B_h(\bm{S},\bm{v})+b_h^*(r_h,\bm{v})+b_h(\bm{w},q).
\end{equation*}
Observe that finding the solution to \eqref{eq:discform} is equivalent to finding a fixed-point $\bm{u}_h$ of $T_h$ so that
\begin{equation*}
    T_h(\bm{u}_h)=\bm{u}_h
\end{equation*}
with its corresponding $\bm{G}_h$ and $p_h$.

\begin{lemma}\label{lem:Th-divfree}
    $T_h$ is well-defined on $V_h^\tdiv$ and $T_h(\bm{z}_h)\in V_h^\tdiv$ for all $\bm{z}_h\in V_h^\tdiv$.
\end{lemma}
\begin{proof}
    Let $\bm{z}_h\in V_h^\tdiv$ be given.
    From Lemma~\ref{lem:divfree}, it is clear that $T_h(\bm{z}_h)\in V_h^\tdiv$.
    By Lemma~\ref{lem:Nh-nonnegativity} and the definition of $A_h$, we obtain
    \begin{equation*}
        A_h(\varphi,\varphi)+N_h(\bm{z}_h;\bm{v},\bm{v})\geq \nu^{-1}\norm{\bm{H}}_{L^2(\Omega)}^2\quad\forall \varphi=(\bm{H},\bm{v},q)\in H_h\times V_h\times Q_h^0.
    \end{equation*}
    Let $\varsigma_i=(\bm{G}_i,\bm{w}_i,p_i)$ be the solution to \eqref{eq:Th} with $\bm{z}_h$ for $i=1,2$.
    Then
    \begin{equation*}
        \nu^{-1}\norm{\bm{G}_1-\bm{G}_2}_{L^2(\Omega)}^2\leq  A_h(\varsigma_1-\varsigma_2,\varsigma_1-\varsigma_2)+N_h(\bm{z}_h;\bm{w}_1-\bm{w}_2,\bm{w}_1-\bm{w}_2)=0.
    \end{equation*}
    Therefore, we have $\bm{G}_1=\bm{G}_2$.
    Furthermore, Lemma~\ref{lem:infsup} and taking $(\bm{v},q)=(\bm{0},0)$ in \eqref{eq:Th} lead to
    \begin{equation*}
        \norm{\bm{w}_1-\bm{w}_2}_h\leq C\sup_{\bm{H}\in H_h}\frac{B_h(\bm{H},\bm{w}_1-\bm{w}_2)}{\norm{\bm{H}}_{L^2(\Omega)}}=C\sup_{\bm{H}\in H_h}\frac{B_h^*(\bm{w}_1-\bm{w}_2,\bm{H})}{\norm{\bm{H}}_{L^2(\Omega)}}=0.
    \end{equation*}
    Finally, the discrete adjoint property \eqref{eq:adjoint} and \eqref{eq:Th} with $(\bm{H},q)=(\bm{0},0)$ read
    \begin{equation*}
        0=b_h^*(p_1-p_2,\bm{v})=b_h(\bm{v},p_1-p_2)\quad\forall \bm{v}\in V_h.
    \end{equation*}
    By the inf-sup condition \eqref{lem:infsup}, we can obtain
    \begin{align*}
        \norm{p_1-p_2}_{1,h}\leq C \sup_{\bm{v}\in V_h}\frac{b_h(\bm{v},p_1-p_2)}{\norm{\bm{v}}_{L^2(\Omega)}}=0.
    \end{align*}
    Therefore, we have $\varsigma_1=\varsigma_2$.
    Since $V_h$ is finite dimensional space and $T_h$ is linear, $T_h$ is well-defined.

\end{proof}

Then, we have the following stability estimate.
\begin{lemma}\label{lem:Thbound}
    For any $\bm{z}_h\in V_h^\tdiv$, we have
    \begin{equation*}
        \norm{T_h(\bm{z}_h)}_h=\norm{\bm{w}_h}_h\leq C_B^{-2}\nu^{-1}\norm{\bm{f}_0}_{L^2(\Omega)},
    \end{equation*}
    where $\bm{f}_0$ is from the Helmholtz-Hodge decomposition $\bm{f}=\bm{f}_0+\nabla\chi$.
\end{lemma}
\begin{proof}
    Let $\bm{z}_h\in V_h^\tdiv$ and $\bm{w}_h=T_h(\bm{z}_h)$. Then there exists $\varsigma=(\bm{G}_h,\bm{w}_h,p_h)$ such that
    \begin{equation*}
        A_h(\varsigma,\varphi)+N_h(\bm{z}_h;\bm{w}_h,\bm{v})=(\bm{f},\bm{v})\quad\forall\varphi=(\bm{H},\bm{v},q)\in H_h\times V_h\times Q_h^0.
    \end{equation*}
    Taking $\varphi=\varsigma$, we obtain
    \begin{equation*}
        \nu^{-1}\norm{\bm{G}_h}_{L^2(\Omega)}^2\leq A_h(\varsigma,\varsigma)+N_h(\bm{z}_h;\bm{w}_h,\bm{w}_h)=(\bm{f},\bm{w}_h).
    \end{equation*}
    From Lemma~\ref{lem:infsup} and \eqref{eq:Th}, we have
    \begin{equation*}
        C_B^2\norm{\bm{w}_h}_h^2\leq \nu^{-2}\norm{\bm{G}_h}_{L^2(\Omega)}^2\leq \nu^{-1}(\bm{f},\bm{w}_h).
    \end{equation*}
    Finally, recall that $\bm{w}_h=T_h(\bm{z}_h)\in V_h^\tdiv$ by Lemma~\ref{lem:Th-divfree}.
    Then the Helmholtz-Hodge decomposition $\bm{f}=\bm{f}_0+\nabla\chi$ yields
    \begin{equation*}
        C_B^2\norm{\bm{w}_h}_h\leq \nu^{-1}\frac{(\bm{f}_0 +\nabla \chi,\bm{w}_h)}{\norm{\bm{w}_h}_h}=\nu^{-1}\frac{(\bm{f}_0,\bm{w}_h)}{\norm{\bm{w}_h}_h}\leq \nu^{-1}\norm{\bm{f}_0}_{L^2(\Omega)}.
    \end{equation*}
    Therefore, the proof is completed.
\end{proof}

\begin{remark}
    Lemma~\ref{lem:Thbound} implies that $T_h(\bm{z}_h)$ is independent of the irrotational part of $\bm{f}$.
\end{remark}

By Lemma~\ref{lem:Thbound} and the Brouwer fixed point theorem, the existence of the fixed-point $\bm{u}_h=T_h(\bm{u}_h)$ is guaranteed.
To show that the fixed-point is unique, it suffices to show that $T_h$ is a contraction mapping.
\begin{theorem}\label{thm:existence}
    Assume that $\bm{f}=\bm{f}_0+\nabla\chi$ satisfies
    \begin{equation*}
        \norm{\bm{f}_0}_{L^2(\Omega)}<C_B^{4}C_N^{-1}\nu^{2}.
    \end{equation*}
    Then $T_h$ has a unique fixed-point in
    \begin{equation*}
        W_h^\rho=\{\bm{v}\in V_h^\tdiv:\norm{\bm{v}}_h\leq \rho<C_B^{2}C_N^{-1}\nu\}.
    \end{equation*}
\end{theorem}
\begin{proof}
    Let $\bm{\varsigma}_i=(\bm{G}_i,\bm{w}_i,p_i)$ be the solution to \eqref{eq:Th} with $\bm{z}_i\in V_h^\tdiv$ for $i=1,2$.
    Then we have
    \begin{equation*}
        A_h(\bm\varsigma_i,\bm\varphi)+N_h(\bm{z}_i;\bm{w}_i,\bm{v})=(\bm{f},\bm{v})\quad\forall\varphi=(\bm{H},\bm{v},q)\in H_h\times V_h\times Q_h^0.
    \end{equation*}
    Taking $\bm\varphi=\bm\varsigma_1-\bm\varsigma_2$, we obtain
    \begin{equation*}
        \begin{aligned}
            \nu^{-1}\norm{\bm{G}_1-\bm{G}_2}_{L^2(\Omega)}^2
            &\leq A_h(\bm\varsigma_1-\bm\varsigma_2,\bm\varsigma_1-\bm\varsigma_2) + N_h(\bm{z}_1;\bm{w}_1-\bm{w}_2,\bm{w}_1-\bm{w}_2)\\
            &=N_h(\bm{z}_2-\bm{z}_1;\bm{w}_2,\bm{w}_1-\bm{w}_2)\\
            &\leq C_N\norm{\bm{z}_1-\bm{z}_2}_h\norm{\bm{w}_2}_h\norm{\bm{w}_1-\bm{w}_2}_h.
        \end{aligned}
    \end{equation*}
    From Lemma~\ref{lem:infsup} and Lemma~\ref{lem:Thbound}, we obtain
    \begin{equation*}
        \begin{aligned}
        \norm{\bm{w}_1-\bm{w}_2}_h^2
        &\leq C_B^{-2}\nu^{-2}\norm{\bm{G}_1-\bm{G}_2}_{L^2(\Omega)}^2\\
        &\leq C_B^{-2}C_N\nu^{-1}\norm{\bm{z}_2-\bm{z}_1}_h\norm{\bm{w}_2}\norm{\bm{w}_1-\bm{w}_2}_h\\
        &\leq C_B^{-4}C_N\nu^{-2}\norm{\bm{f}_0}_{L^2(\Omega)}\norm{\bm{z}_2-\bm{z}_1}_h\norm{\bm{w}_1-\bm{w}_2}_h.
        \end{aligned}
    \end{equation*}
    Recalling that $\bm{w}_i=T_h(\bm{z}_i)$, there holds
    \begin{align*}
        \norm{T_h(\bm{z}_1)-T_h(\bm{z}_2)}_h\leq C_B^{-4}\nu^{-2}\norm{\bm{f}_0}_{L^2(\Omega)}\norm{\bm{z}_2-\bm{z}_1}_h.
    \end{align*}
    By assuming $\norm{\bm{f}_0}_{L^2(\Omega)}<C_B^{4}C_N^{-1}\nu^{2}$, $T_h$ is a contraction mapping.
    Therefore, an application of the Brouwer fixed-point theorem implies that there exists a unique fixed-point in $W_h^\rho$.
\end{proof}

\section{\textit{A priori} error estimates}\label{sec:pri}

In this section, we derive \textit{a priori} error estimates of the discrete solution.
In particular, the $L^2$ error estimates for all the variables are derived, and a superconvergent result for velocity can be designed under $\|\cdot\|_h$ norm.

Firstly, we have the following \textit{a priori} bounds for $\norm{\bm{u}_h-J_h\bm{u}}_h$.
\begin{lemma}\label{lem:bound_intpU_by_intpG}
    Let $(\bm{u},\bm{G})$ be the solution of \eqref{eq:model-conserve} and $(\bm{u}_h,\bm{G}_h)\in V_h\times H_h$ be the discrete solution of \eqref{eq:discform}.
    Then the following estimate holds:
    \begin{equation*}
        \norm{J_h\bm{u}-\bm{u}_h}_h\leq C_B\nu^{-1}\norm{\bm{G}-\bm{G}_h}_{L^2(\Omega)}.
    \end{equation*}
\end{lemma}
\begin{proof}
    The Galerkin orthogonality reads
    \begin{equation*}
        \nu^{-1}(\bm{G}-\bm{G}_h,\bm{H})+B_h^*(\bm{u}-\bm{u}_h,\bm{H})=0\quad\forall\bm{H}\in H_h.
    \end{equation*}
    By Lemma~\ref{lem:intp}, we have
    \begin{equation*}
        B_h(\bm{H},J_h\bm{u}-\bm{u}_h)=-B_h^*(\bm{u}-\bm{u}_h,\bm{H})=\nu^{-1}(\bm{G}-\bm{G}_h,\bm{H}).
    \end{equation*}
    This and Lemma~\ref{lem:infsup} imply that
    \begin{equation*}
        \begin{aligned}
        \norm{J_h\bm{u}-\bm{u}_h}_h
        &\leq C_B\sup_{\bm{H}\in H_h}\frac{B_h(\bm{H},J_h\bm{u}-\bm{u}_h)}{\norm{\bm{H}}_{L^2(\Omega)}}\\
        &=C_B\sup_{\bm{H}\in H_h}\frac{\nu^{-1}(\bm{G}-\bm{G}_h,\bm{H})}{\norm{\bm{H}}_{L^2(\Omega)}}\\
        &\leq C_B\nu^{-1}\norm{\bm{G}-\bm{G}_h}_{L^2(\Omega)}.
        \end{aligned}
    \end{equation*}
\end{proof}

We decompose the velocity error $\bm{u}-\bm{u}_h$ into two parts,
\begin{equation*}
    \bm{u}-\bm{u}_h=(\bm{u}-J_h\bm{u})+(J_h\bm{u}-\bm{u}_h)=:\bm{e}_*+\bm{e}_h.
\end{equation*}
\begin{lemma}\label{lem:Nh-uu-uhuh}
    Let $\bm{f}=\bm{f}_0+\nabla\chi$ be the Helmholtz-Hodge decomposition of $\bm{f}$.
    Then we have
    \begin{equation*}
       | N_h(\bm{u};\bm{u},\bm{v}_h)-N_h(\bm{u}_h;\bm{v}_h,\bm{v}_h)|\leq \nu^{-1}\norm{\bm{f}_0}_{L^2(\Omega)}(C_{-1}\norm{\bm{e}_*}_{0,4,h}+C_*\norm{\bm{e}_h}_h)\norm{\bm{v}_h}_h
    \end{equation*}
    for all $\bm{v}_h\in V_h$.
\end{lemma}
For the proof, see Appendix~\ref{sec:app}.
With the aid of the aforementioned lemmas, we can prove the convergence of $\bm{G}_h$.

\begin{theorem}\label{thm:convG}
    Assume that $\norm{\bm{f}_0}_{L^2(\Omega)}\leq \nu^2\min\{\frac{1}{2}C_*^{-1}C_B^2,C_B^{4}C_N^{-1}\}$.
    Then we have
    \begin{equation*}
        \norm{\bm{G}_h-\Pi_h\bm{G}_h}_{L^2(\Omega)}\leq C(\nu^{2}\norm{\bm{u}-J_h\bm{u}}_{0,4,h}+\norm{\bm{G}-\Pi_h\bm{G}}_{L^2(\Omega)}).
    \end{equation*}
\end{theorem}
\begin{proof}
    Let $(\bm{G}_h,\bm{u}_h,p_h)\in H_h\times V_h\times Q_h^0$ be the solution to \eqref{eq:discform}.
    Then we have
    \begin{equation*}
        \begin{aligned}
            \nu^{-1}\norm{\bm{G}_h-\Pi_h\bm{G}_h}_{L^2(\Omega)}^2
            &=\nu^{-1}(\bm{G}_h-\bm{G},\bm{G}_h-\Pi_h\bm{G})+\nu^{-1}(\bm{G}-\Pi_h\bm{G},\bm{G}_h-\Pi_h\bm{G})\\
            &=B_h^*(\bm{u}_h-\bm{u},\bm{G}_h-\Pi_h\bm{G})
            +\nu^{-1}(\bm{G}-\Pi_h\bm{G},\bm{G}_h-\Pi_h\bm{G}).
        \end{aligned}
    \end{equation*}
    By Lemma~\ref{lem:intp} and the adjoint properties of $B_h(\cdot,\cdot)$ and $B_h^*(\cdot,\cdot)$, we obtain
    \begin{equation*}
        \begin{aligned}
            \nu^{-1}\norm{\bm{G}_h-\Pi_h\bm{G}_h}_{L^2(\Omega)}^2
            &=B_h^*(\bm{u}_h-J_h\bm{u},\bm{G}_h-\Pi_h\bm{G})+\nu^{-1}(\bm{G}-\Pi_h\bm{G},\bm{G}_h-\Pi_h\bm{G})\\
            &=B_h(\bm{G}_h-\Pi_h\bm{G},\bm{u}_h-J_h\bm{u})+\nu^{-1}(\bm{G}-\Pi_h\bm{G},\bm{G}_h-\Pi_h\bm{G})\\
            &=B_h(\bm{G}_h-\bm{G},\bm{u}_h-J_h\bm{u})+\nu^{-1}(\bm{G}-\Pi_h\bm{G},\bm{G}_h-\Pi_h\bm{G}).
        \end{aligned}
    \end{equation*}
    The consistency of discrete operators and \eqref{eq:discform2} yield
    \begin{equation*}
        \begin{aligned}
            \nu^{-1}\norm{\bm{G}_h-\Pi_h\bm{G}_h}_{L^2(\Omega)}^2
            &=N_h(\bm{u};\bm{u},\bm{u}_h-J_h\bm{u})-N_h(\bm{u}_h;\bm{u}_h,\bm{u}_h-J_h\bm{u})\\
            &\quad +b_h^*(p-p_h,\bm{u}_h-J_h\bm{u})+\nu^{-1}(\bm{G}-\Pi_h\bm{G},\bm{G}_h-\Pi_h\bm{G}).
        \end{aligned}
    \end{equation*}
    Lemmas~\ref{lem:intp},\;\ref{lem:Nh-uu-uhuh}, \eqref{eq:adjoint} and \eqref{eq:discform3} lead to
    \begin{equation*}
        \begin{aligned}
            \nu^{-1}\norm{\bm{G}_h-\Pi_h\bm{G}_h}_{L^2(\Omega)}^2
            &=N_h(\bm{u};\bm{u},\bm{u}_h-J_h\bm{u})-N_h(\bm{u}_h;\bm{u}_h,\bm{u}_h-J_h\bm{u})\\
            &\quad +b_h(\bm{u}_h-J_h\bm{u},I_hp-p_h)+\nu^{-1}(\bm{G}-\Pi_h\bm{G},\bm{G}_h-\Pi_h\bm{G})\\
            &=N_h(\bm{u};\bm{u},\bm{u}_h-J_h\bm{u})-N_h(\bm{u}_h;\bm{u}_h,\bm{u}_h-J_h\bm{u})\\
            &\quad +\nu^{-1}(\bm{G}-\Pi_h\bm{G},\bm{G}_h-\Pi_h\bm{G})\\
            &\leq C\nu^{-1}(\norm{\bm{f}_0}_{L^2(\Omega)}\norm{\bm{u}-J_h\bm{u}}_{0,4,h} + \norm{\bm{G}-\Pi_h\bm{G}}_{L^2(\Omega)})\norm{\bm{G}_h-\Pi_h\bm{G}}_{L^2(\Omega)}\\
            &\quad +C_*\nu^{-1}\norm{\bm{f}_0}_{L^2(\Omega)}\norm{J_h\bm{u}-\bm{u}_h}_h^2.
        \end{aligned}
    \end{equation*}
    Using Lemma~\ref{lem:bound_intpU_by_intpG}, we obtain
    \begin{equation*}
        \nu^{-1}(1-C_*C_B^{-2}\nu^{-2}\norm{\bm{f}_0}_{L^2(\Omega)})\norm{\bm{G}_h-\Pi_h\bm{G}_h}_{L^2(\Omega)}\leq C\nu^{-1}(\norm{\bm{f}_0}_{L^2(\Omega)}\norm{\bm{u}-J_h\bm{u}}_{0,4,h}+\norm{\bm{G}-\Pi_h\bm{G}}_{L^2(\Omega)}).
    \end{equation*}
    This coupled with the assumption $\norm{\bm{f}_0}_{L^2(\Omega)}\leq \frac{1}{2}C_*^{-1}C_B^2\nu^2$ yields
    \begin{equation*}
        \norm{\bm{G}_h-\Pi_h\bm{G}}_{L^2(\Omega)}\leq C(\norm{\bm{f}_0}_{L^2(\Omega)}\norm{\bm{u}-J_h\bm{u}}_{0,4,h} + \norm{\bm{G}-\Pi_h\bm{G}}_{L^2(\Omega)}).
    \end{equation*}
    Therefore, the proof is completed.
\end{proof}
Now, we are ready to prove the main theorem of this paper.
\begin{theorem}\label{thm:conv}
    Assume that $\bm{f}_0$ satisfies the assumption in Theorem~\ref{thm:convG} and the weak solution $(\bm{u},p)$ to \eqref{eq:model-conserve} belongs to $[H^{k+2}(\Omega)]^2\times H^{k+1}(\Omega)$.
    Then it holds
    \begin{equation*}
        \begin{aligned}
            \norm{\bm{G}-\bm{G}_h}_{L^2(\Omega)}&\leq C\nu h^{k+1}\norm{\bm{u}}_{H^{k+2}(\Omega)},\\
            \norm{\bm{u}-\bm{u}_h}_{L^2(\Omega)}&\leq Ch^{k+1}\norm{\bm{u}}_{H^{k+2}(\Omega)},\\
            \norm{p-p_h}_{L^2(\Omega)}&\leq Ch^{k+1}(\nu \norm{\bm{u}}_{H^{k+2}(\Omega)}+\norm{p}_{H^{k+1}(\Omega)}).
        \end{aligned}
    \end{equation*}
    In addition, the following superconvergence holds
    \begin{align*}
      \norm{J_h\bm{u}-\bm{u}_h}_h&\leq Ch^{k+1}\norm{\bm{u}}_{H^{k+2}(\Omega)}.
    \end{align*}

\end{theorem}

\begin{proof}
    From Theorem~\ref{thm:conv} and Lemma~\ref{lem:intperr}, we have
    \begin{equation*}
        \begin{aligned}
        \norm{\bm{G}-\bm{G}_h}_{L^2(\Omega)}
        &\leq \norm{\bm{G}-\Pi_h\bm{G}}_{L^2(\Omega)}+\norm{\bm{G}_h-\Pi_h\bm{G}}_{L^2(\Omega)}\\
        &\leq C(\nu^2\norm{\bm{u}-J_h\bm{u}}_{0,4,h}+\norm{\bm{G}-\Pi_h\bm{G}}_{L^2(\Omega)})\\
        &\leq Ch^{k+1}(\nu^2\norm{\bm{u}}_{W^{k+1,4}(\Omega)}+\norm{\bm{G}}_{H^{k+1}(\Omega)}).
        \end{aligned}
    \end{equation*}
    Note that $\bm{G}=\nu\nabla\bm{u}$, the Sobolev embedding theorem yields
    \begin{equation*}
        \norm{\bm{G}-\bm{G}_h}_{L^2(\Omega)}\leq Ch^{k+1}(\nu^2\norm{\bm{u}}_{W^{k+1,4}(\Omega)}+\nu\norm{\bm{u}}_{H^{k+2}(\Omega)}).
    \end{equation*}
    Then the assumption on $\bm{f}_0$ implies
    \begin{equation*}
        \norm{\bm{G}-\bm{G}_h}_{L^2(\Omega)}\leq C\nu h^{k+1}\norm{\bm{u}}_{H^{k+2}(\Omega)}.
    \end{equation*}
    This and Lemma~\ref{lem:bound_intpU_by_intpG} yield the superconvergence
    \begin{equation*}
        \norm{J_h\bm{u}-\bm{u}_h}_h\leq C\nu^{-1}\norm{\bm{G}-\bm{G}_h}_h\leq Ch^{k+1}\norm{\bm{u}}_{H^{k+2}(\Omega)}.
    \end{equation*}
    The second assertion can be derived by using the triangle inequality and Lemma~\ref{lem:discSobolev}.
    \begin{equation*}
        \begin{aligned}
        \norm{\bm{u}-\bm{u}_h}_{L^2(\Omega)}
        &\leq \norm{\bm{u}-J_h\bm{u}}_{L^2(\Omega)}+\norm{\bm{u}_h-J_h\bm{u}}_{L^2(\Omega)}\\
        &\leq \norm{\bm{u}-J_h\bm{u}}_{L^2(\Omega)}+C\norm{\bm{u}_h-J_h\bm{u}}_h\\
        &\leq Ch^{k+1}\norm{\bm{u}}_{H^{k+2}(\Omega)}.
        \end{aligned}
    \end{equation*}
    Finally, we prove the error estimates for the pressure variable.
    By the inf-sup condition \cite{BraessBook} and integration by parts, the following holds 
    \begin{equation*}
        \begin{aligned}
        C\norm{p_h-p}_{L^2(\Omega)}
        &\leq \sup_{\bm{v}\in H^1_0(\Omega)}\frac{(\nabla\cdot\bm{v},p_h-p)}{\norm{\bm{v}}_{H^1(\Omega)}}\\
        &\leq \sup_{\bm{v}\in H^1_0(\Omega)}\frac{(\nabla\cdot\bm{v},p_h-I_hp)}{\norm{\bm{v}}_{H^1(\Omega)}}+\sup_{\bm{v}\in H^1_0(\Omega)}\frac{(\nabla\cdot \bm{v},I_hp-p)}{\norm{\bm{v}}_{H^1(\Omega)}}\\
        &\leq \sup_{\bm{v}\in H^1_0(\Omega)}\frac{-b_h(\bm{v},p_h-I_hp)}{\norm{\bm{v}}_{H^1(\Omega)}}+Ch^{k+1}\norm{p}_{H^{k+1}(\Omega)}.
        \end{aligned}
    \end{equation*}
    Therefore, it is enough to bound the first term on the right-hand side.
    Using the definition and the stability estimate of $J_h$, we obtain
    \begin{equation*}
        \sup_{\bm{v}\in H^1_0(\Omega)}\frac{-b_h(\bm{v},p_h-I_hp)}{\norm{\bm{v}}_{H^1(\Omega)}}\leq C\sup_{\bm{v}\in H^1_0(\Omega)}\frac{-b_h(J_h\bm{v},p_h-I_hp)}{\norm{J_h\bm{v}}_h}.
    \end{equation*}
    The error equation reads
    \begin{equation*}
        \sup_{\bm{v}\in H^1_0(\Omega)}\frac{-b_h(\bm{v},p_h-I_hp)}{\norm{\bm{v}}_{H^1(\Omega)}}\leq C\sup_{\bm{v}\in H^1_0(\Omega)}\frac{B_h(\bm{G}_h-\bm{G},J_h\bm{v})+N_h(\bm{u}_h;\bm{u}_h,J_h\bm{v})-N_h(\bm{u};\bm{u},J_h\bm{v})}{\norm{J_h\bm{v}}_h}.
    \end{equation*}
    Combining Lemma~\ref{lem:Nh-uu-uhuh} and the above estimates completes the proof.
\end{proof}

\begin{remark}
    The usual finite element methods for the (Navier-)Stokes problem yields velocity error estimates that depend on $\nu^{-1}\norm{p-I_hp}_0$.
    Thanks to the special finite element pairs designed, our method yields divergence free velocity, thereby the velocity error estimate is independent of $\nu$.
    The error estimates for velocity gradient and pressure also depend on the positive power of $\nu$, which is desirable when $\nu\ll 1$.
\end{remark}

\section{Numerical Experiments}\label{sec:num}
In this section, we verify our theoretical findings and demonstrate the performance of the proposed method. In particular, the pressure-robustness and optimality of the proposed method will be illustrated.
To solve the nonlinear equation, we used the fixed-point iteration: For $n=1,2,\cdots$, find $(\bm{G}_h^n,\bm{u}_h^n,p_h^n)\in H_h\times V_h\times Q_h$ such that
\begin{equation*}
    \begin{aligned}
        \nu^{-1}(\bm{G}_h^n,\bm{H}) + B_h^*(\bm{u}_h^n,\bm{H})&=\sum_{e\in\mathcal{F}_{pr}^b}\einner{\bm{g},\bm{H}\nn}_e&&\forall \bm H\in H_h,\\
        B_h(\bm{G}_h^n,\bm{v})+N_h(\bm{u}_h^{n-1};\bm{u}_h^n,\bm{v})+b_h^*(p^n,\bm{v})&=(\bm{f},\bm{v})+\sum_{e\in\mathcal{F}_h^b}\einner{|\bm{g}\cdot\nn|,\bm{g}\cdot\bm{v}}_e&&\forall\bm{v}\in V_h,\\
        b_h(\bm{u}_h^n,q)&=\sum_{e\in\mathcal{F}_{pr}^b}\einner{\bm{g}\cdot\nn, q}_e&&\forall q\in Q_h
    \end{aligned}
\end{equation*}
with known $\bm{u}_h^{n-1}$ and the initial guess is chosen to be zero.
The stopping criteria is set to a successive maximum error at nodal points with the tolerance $10^{-7}$.
In this section, $\|\cdot\|$ in the legend of each figure represents $\|\cdot\|_{L^2(\Omega)}$.

\subsection{Taylor vortex}
\begin{figure}
    \centering
    \includegraphics[width=0.45\textwidth]{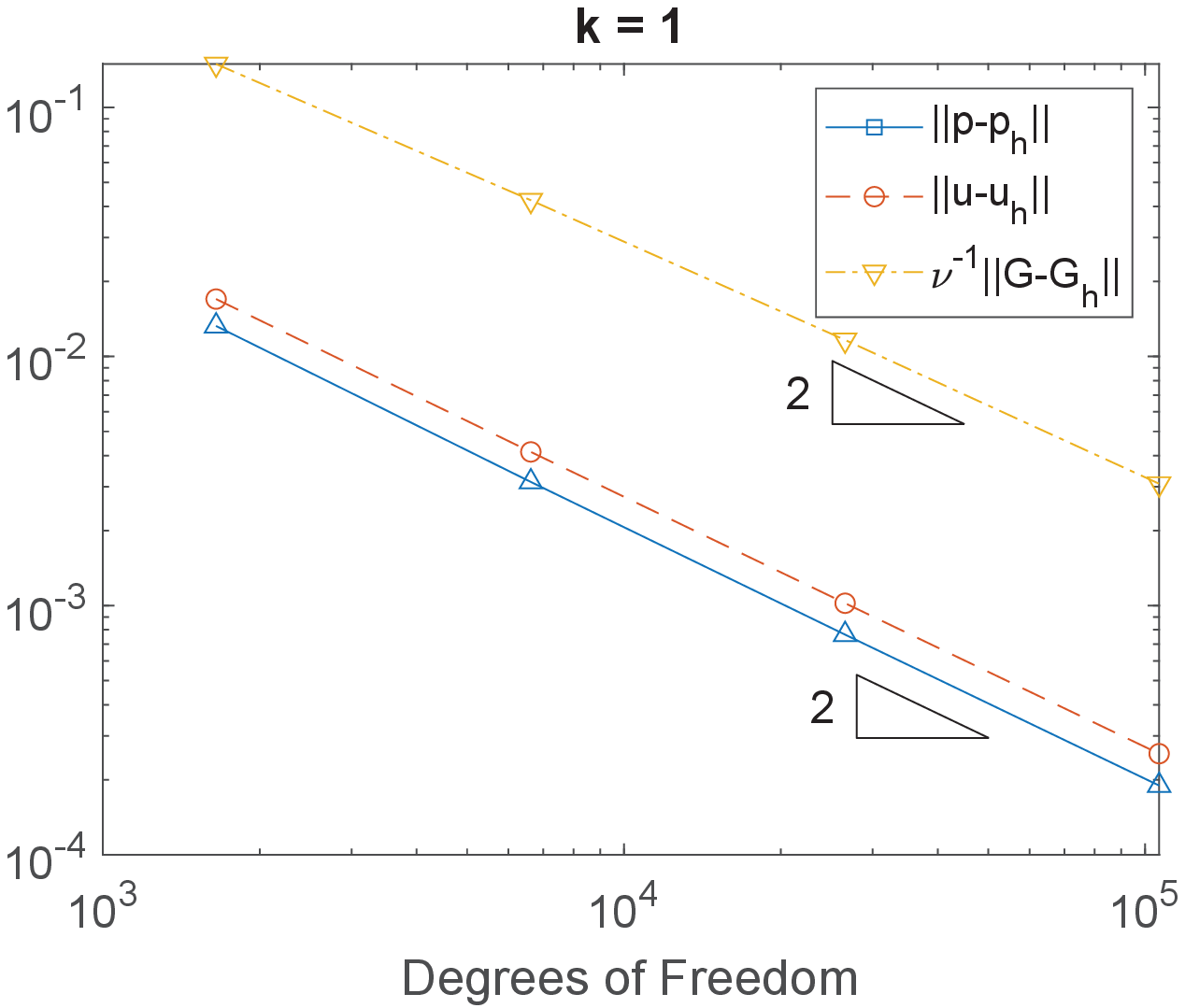}
    \includegraphics[width=0.45\textwidth]{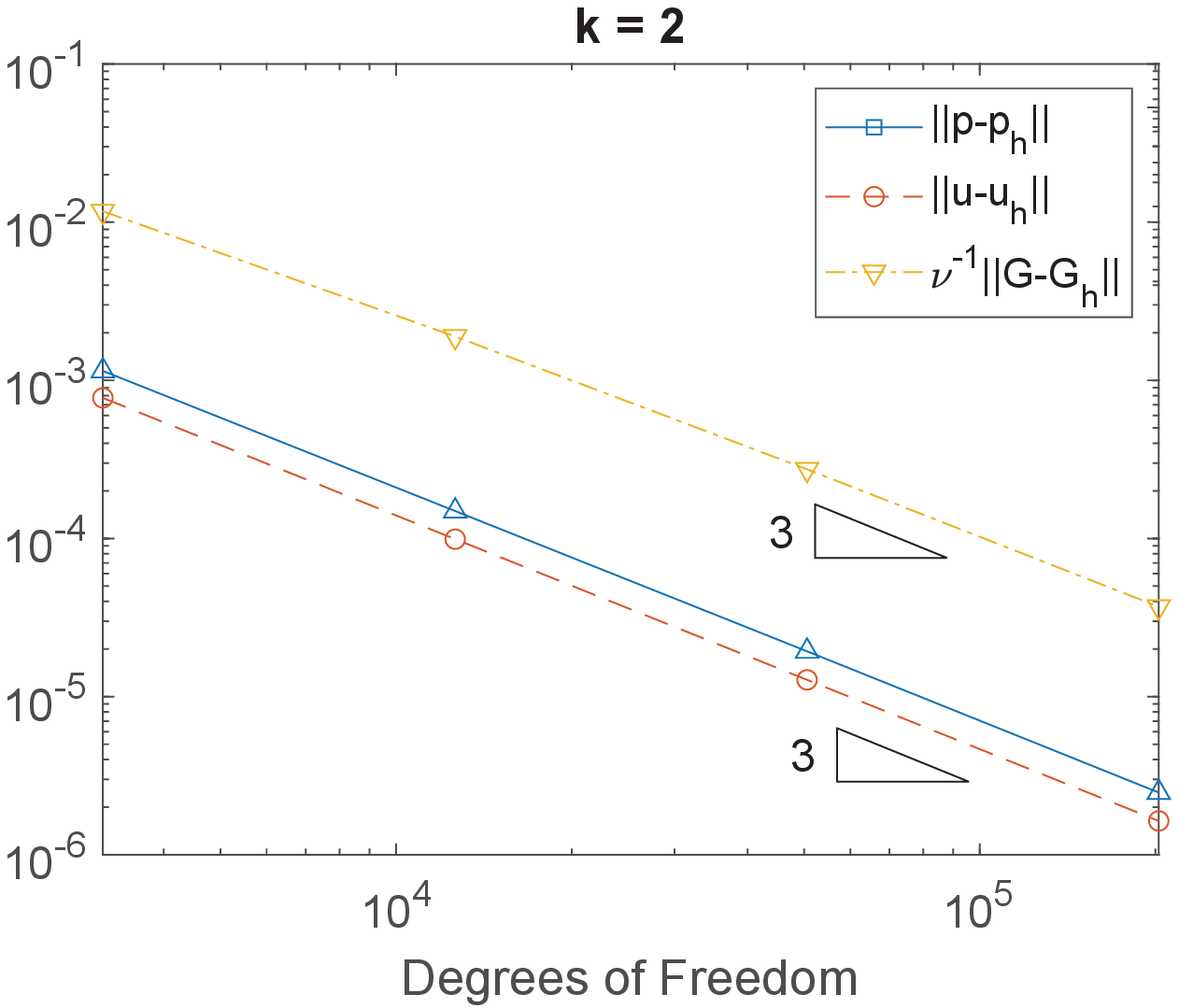}
    \caption{Convergence history with respect to degrees of freedom with $k=1$ (left) and $k=2$ (right) for the Taylor vortex example.}
    \label{fig:conv_taylor_NS}
\end{figure}
\begin{figure}
    \centering
    \includegraphics[width=0.45\textwidth]{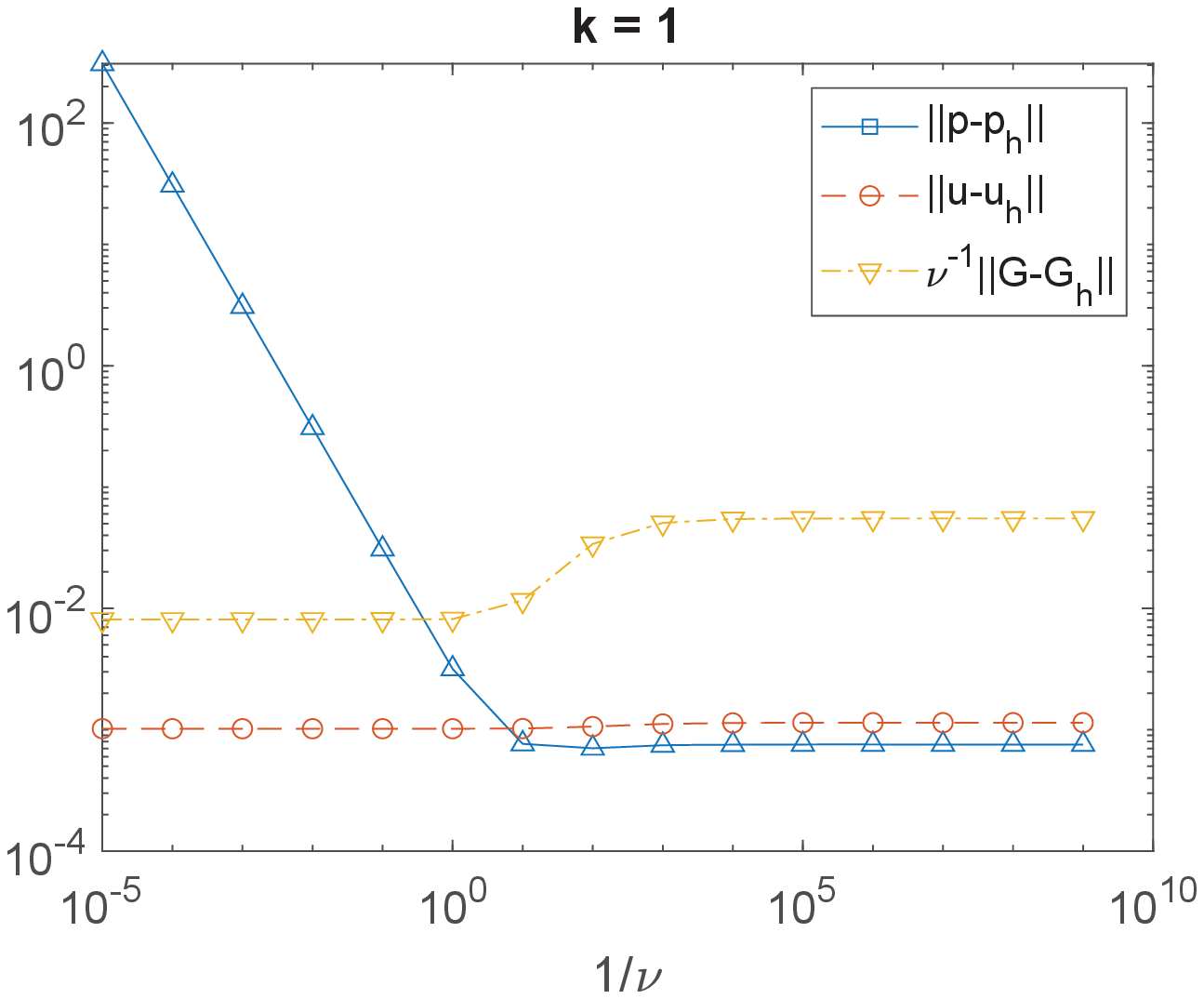}
    \includegraphics[width=0.45\textwidth]{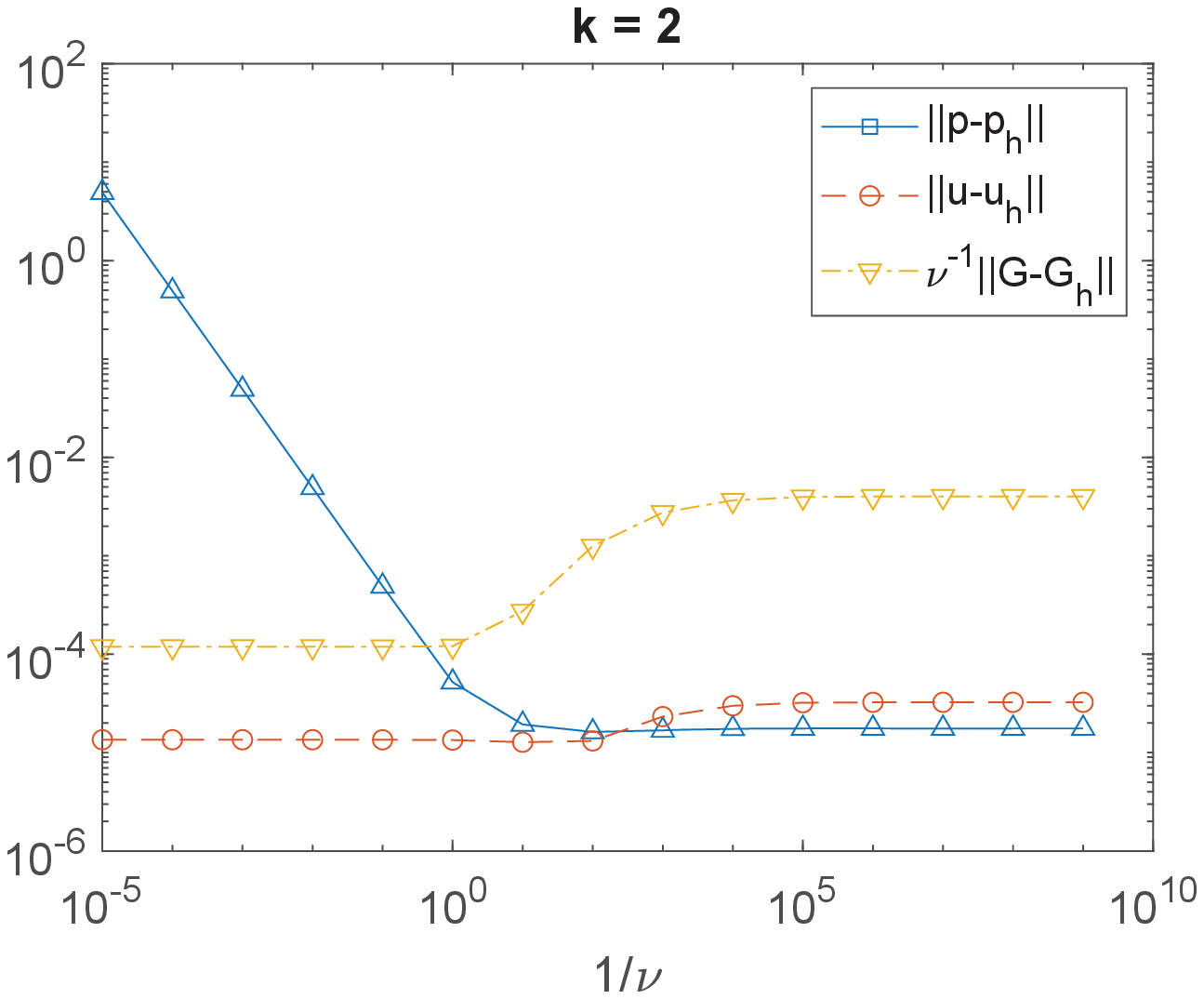}
    \caption{$L^2$-error with various $\nu$ for the Taylor vortex example of the Navier-Stokes equations. Here, $h=1/16$ is used with $k=1$ (left) and $k=2$ (right).}
    \label{fig:taylor_NS}
\end{figure}
\begin{figure}
    \centering
    \includegraphics[width=0.45\textwidth]{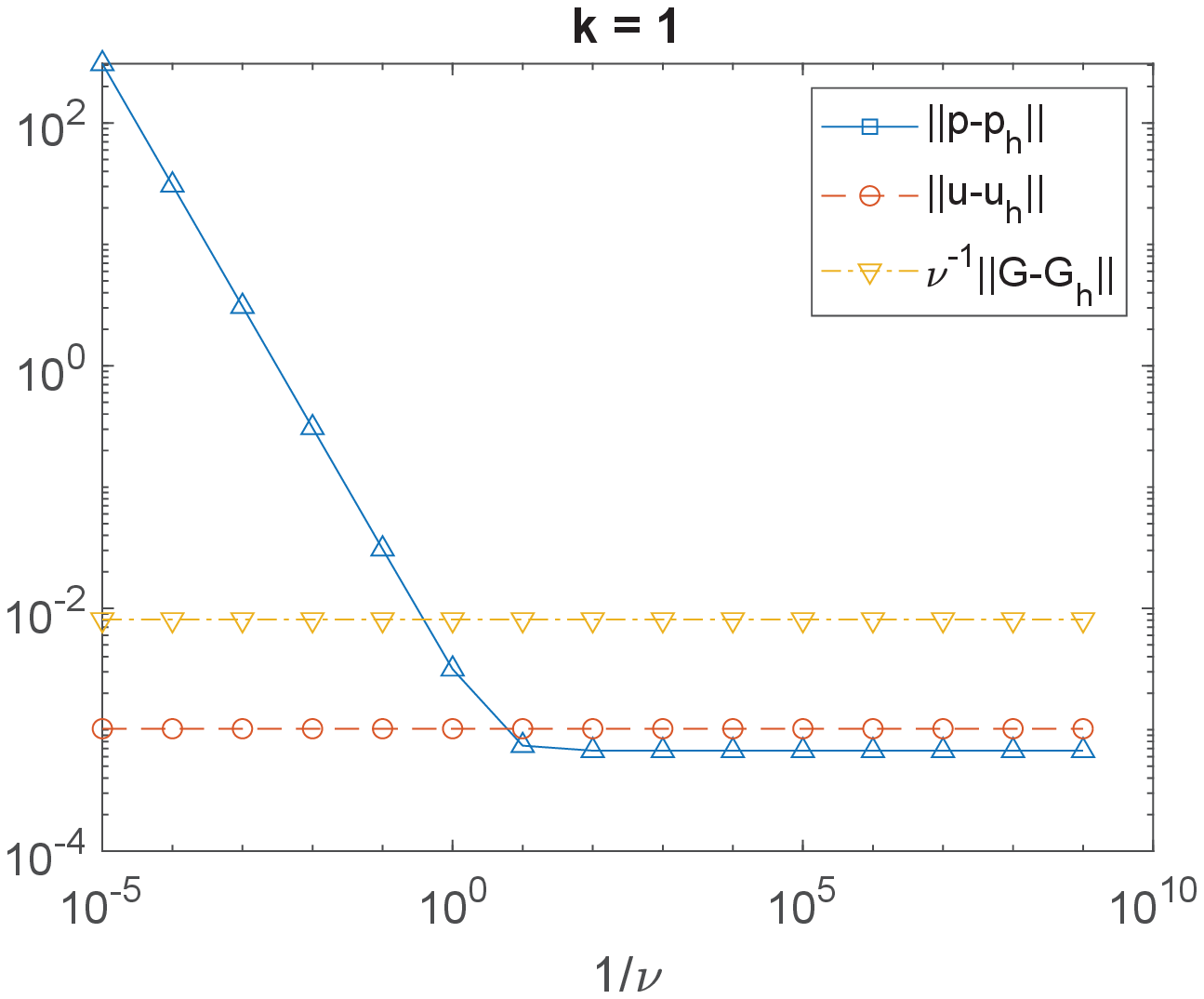}
    \includegraphics[width=0.45\textwidth]{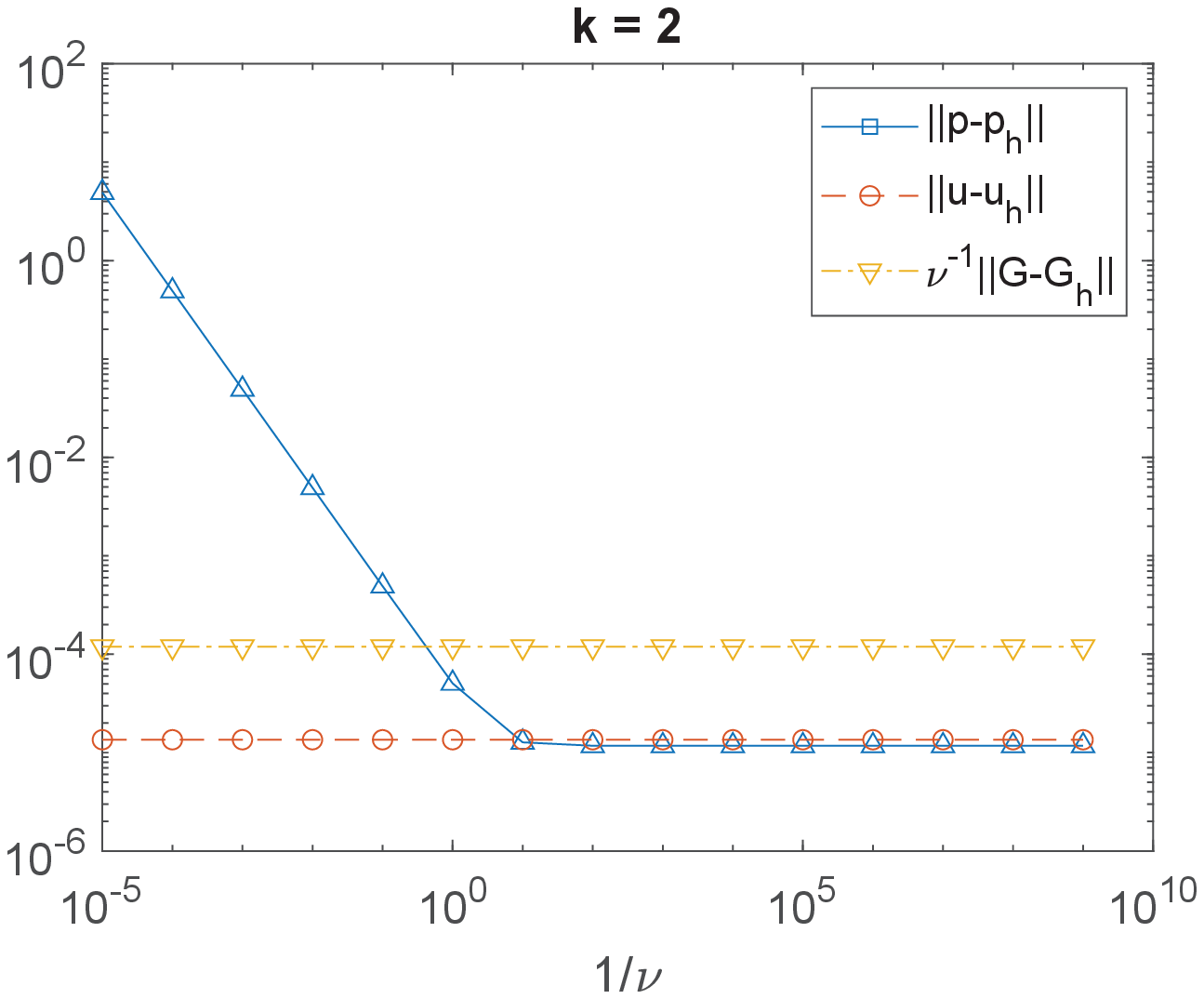}
    \caption{$L^2$-error with various $\nu$ for the Taylor vortex example of the Stokes equations. Here, $h=1/16$ is used with $k=1$ (left) and $k=2$ (right).}
    \label{fig:taylor_Stokes}
\end{figure}

In this example, the optimal convergence and the pressure-robustness of the solution are studied.
Consider a solution $(\bm{u},p)$ defined by
\begin{equation*}
    \bm{u}=\begin{bmatrix}-\cos(\pi x)\sin(\pi y)\\\phantom{-}\sin(\pi x)\cos(\pi y)\end{bmatrix},
    \quad p=-\frac{\cos(2\pi x)+\cos(2\pi y)}{4}.
\end{equation*}
The velocity gradient $\bm{G}$ and the force $\bm{f}$ can be computed from the solution.

In the first experiment, we set $\nu=10^{-1}$.
To observe the convergence behavior of the discrete solution, uniform rectangular grids with $h=2^{-2},\;2^{-3},\;\cdots,\;2^{-5}$ are considered.
In Figure~\ref{fig:conv_taylor_NS}, the convergence history against the number of degrees of freedom with polynomial orders $k=1$ and $k=2$ is depicted.
Optimal convergence rates derived in Theorem~\ref{thm:conv} for all variables are observed for both $k=1$ and $k=2$.

The second experiment is performed to observe the pressure-robustness of the proposed method.
In Theorem~\ref{thm:conv}, both velocity and gradient errors are independent of pressure error.
Also, we expect that the error is bounded by interpolation error independent of $\nu$ when $\bm{f}$ satisfies the small data assumption.
In Figure~\ref{fig:taylor_NS}, $L^2$-errors with varying $\nu$ are observed.
We can observe that the pressure error depends on $\nu$ when $\nu$ is not sufficiently small which can be attributed to the dominant $\nu\norm{\bm{u}-I_h\bm{u}}_{L^2(\Omega)}$.
However, the error shows $\nu$-independent behavior for sufficiently small $\nu$, i.e., the interpolation error of the pressure dominates the pressure error.
For the velocity error, we can observe that the velocity error is independent of $\nu$ when $\nu$ is sufficiently small.
While the convergence analysis without small data assumption for $\bm{f}$ implies that the error may depend on $\nu^{-1}$, the error shows independent behavior for sufficiently small $\nu$.

The dependence of gradient error on $\nu^{-1}$ comes from the nonlinear convective term $N_h(\cdot;\cdot,\cdot)$.
Therefore, we can expect that our formulation leads to $\nu$-independent error for velocity and velocity gradient when the Stokes problem is considered.
In Figure~\ref{fig:taylor_Stokes}, we indeed have $\nu$-independent error behavior for both $k=1$ and $k=2$.

\subsection{No-flow example}
\begin{figure}
    \centering
    \includegraphics[width=0.45\textwidth]{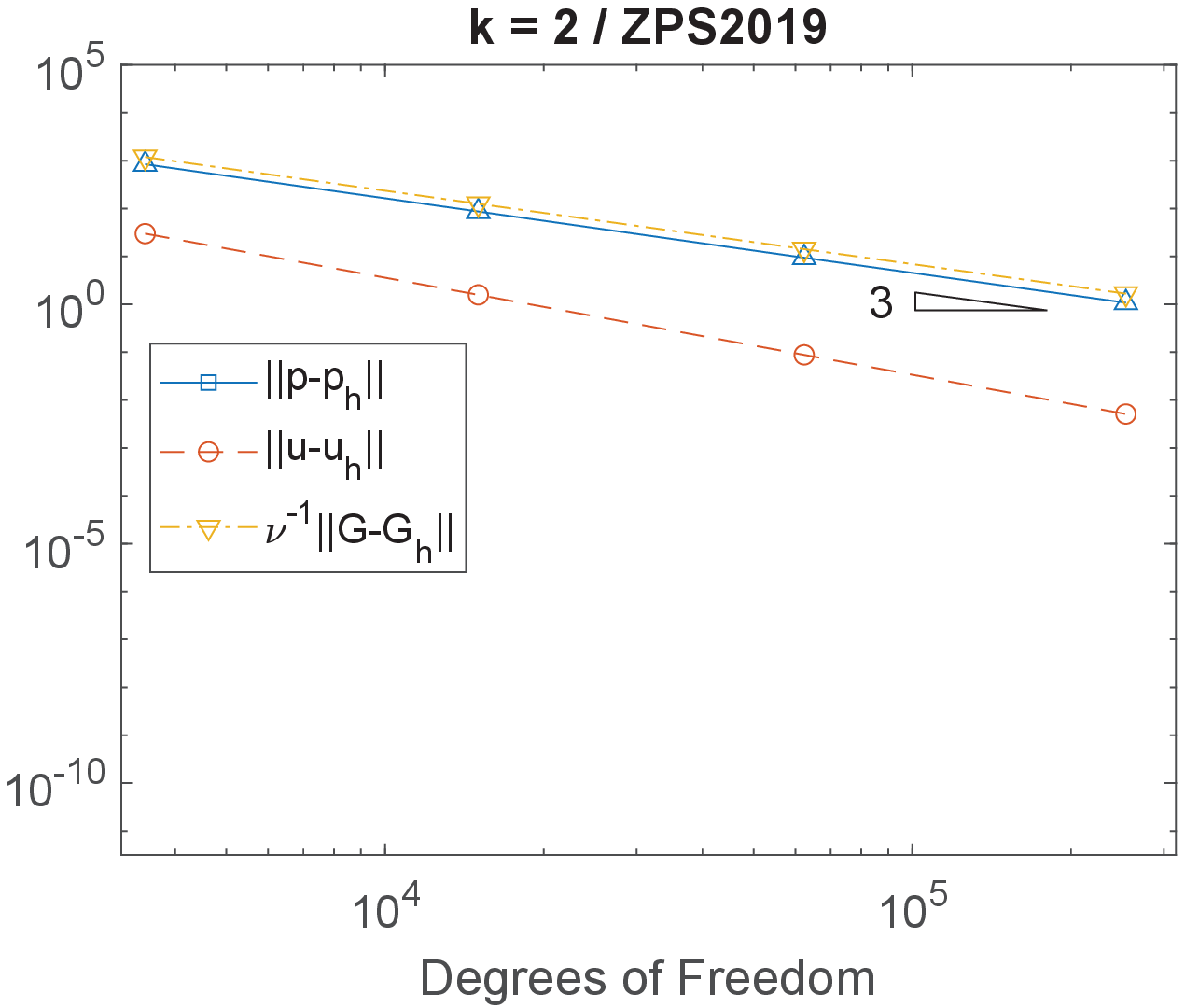}
    \includegraphics[width=0.45\textwidth]{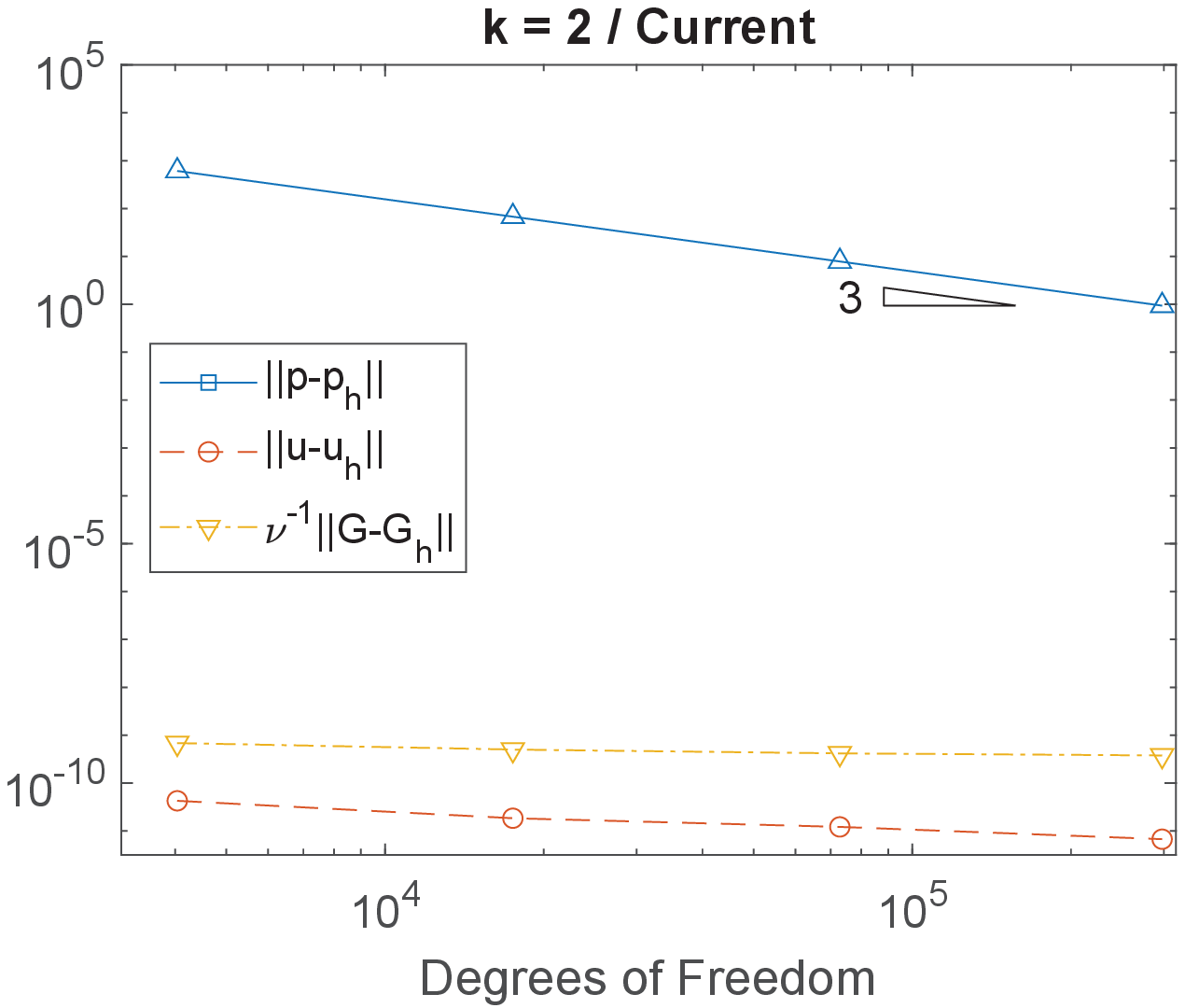}
    \caption{Convergence history of $L^2$-norm for the no-flow example with $\lambda=10^7$ and $\nu=1$. The formulation in \cite{Zhao2019} (left) and the current formulation (right) are used, respectively.}
    \label{fig:noflow_conv}
\end{figure}
\begin{figure}
    \centering
    \includegraphics[width=0.45\textwidth]{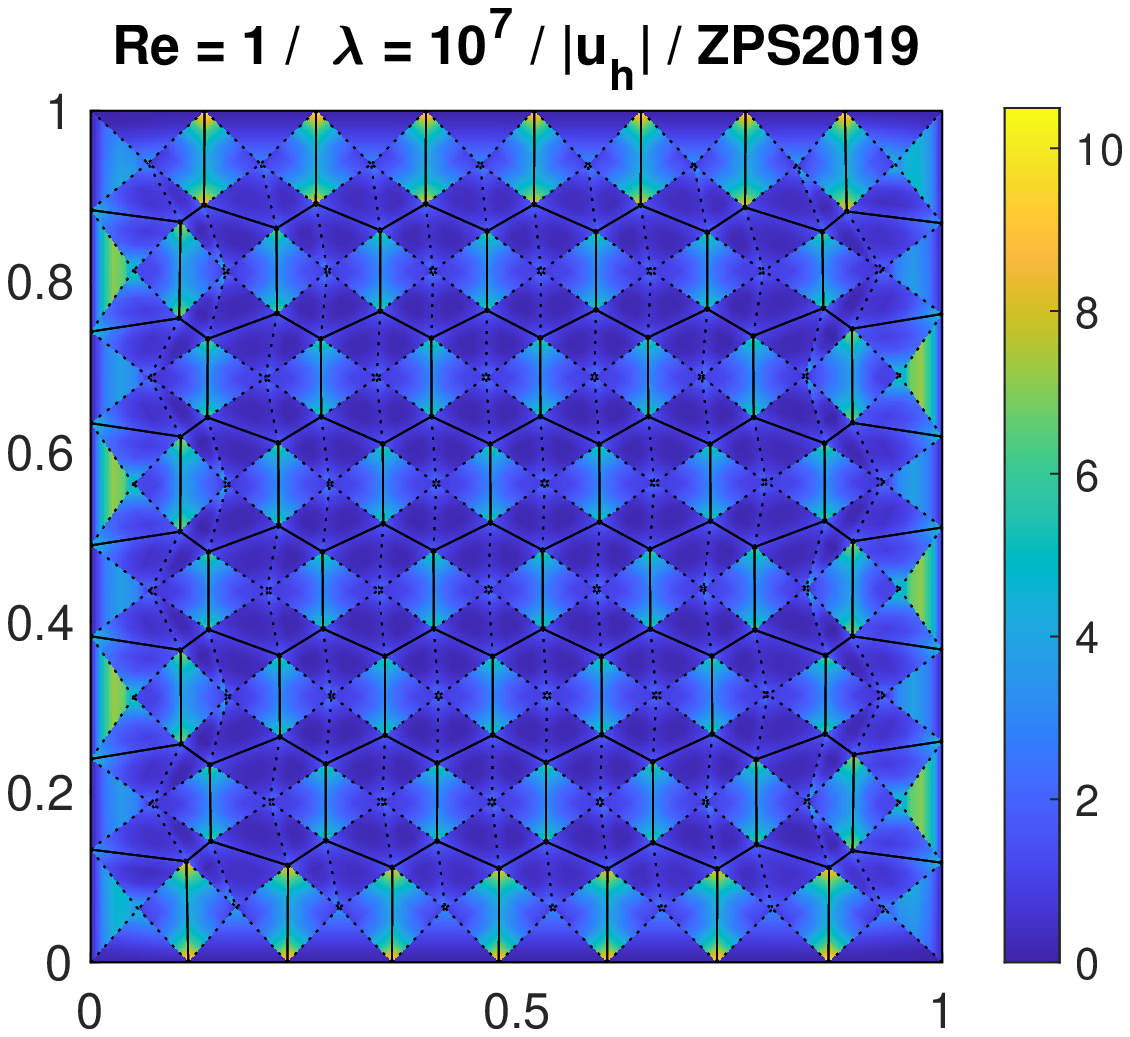}
    \includegraphics[width=0.45\textwidth]{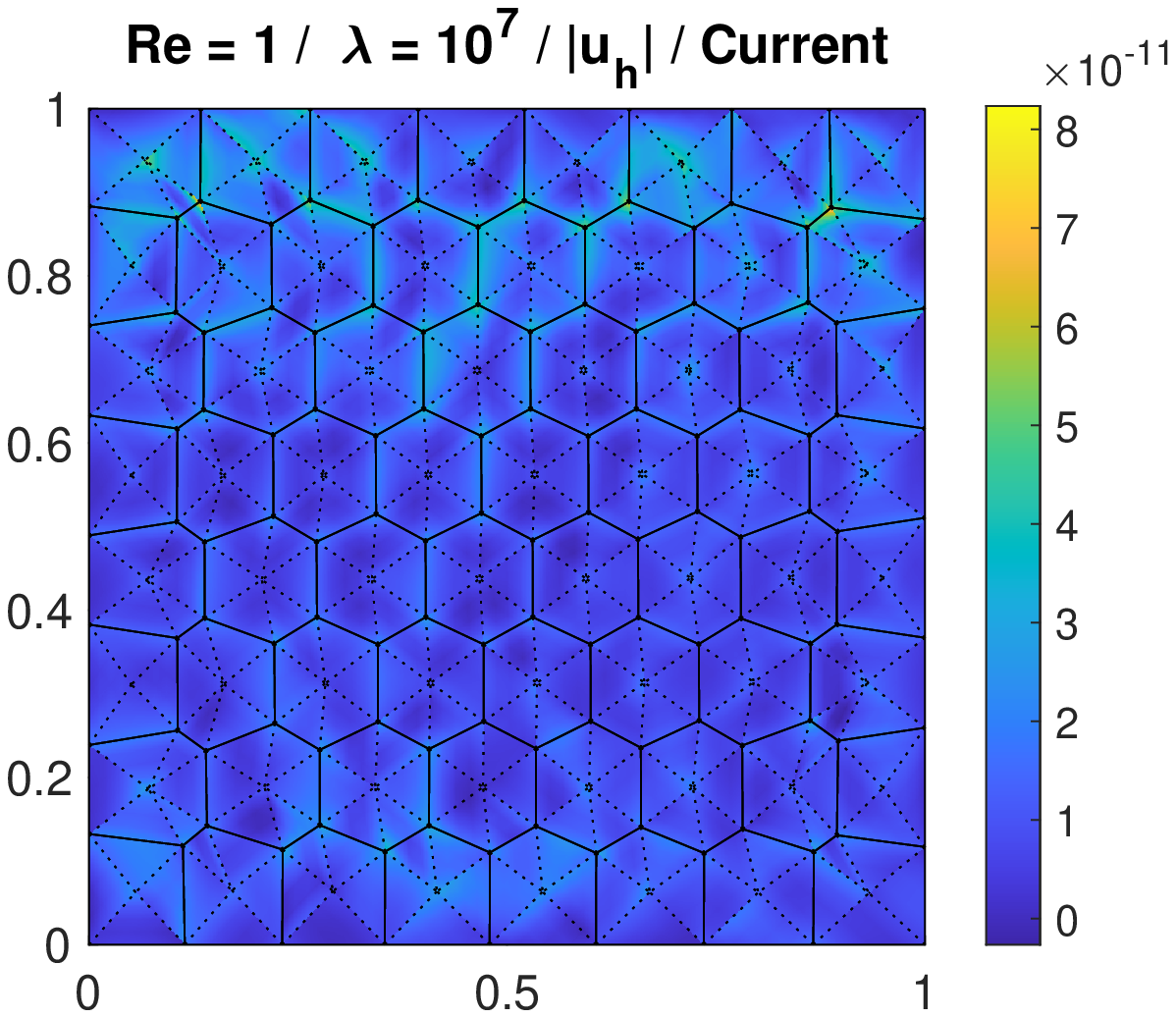}
    \caption{No-flow example. Velocity magnitudes obtained with the formulation given in \cite{Zhao2019} (left) and current (right). Both approximations are obtained with a quasi-uniform polygonal mesh of $h\approx1/8$.}
    \label{fig:noflow}
\end{figure}
In this example, we investigate the pressure-robustness more thoroughly.
Consider the following no-flow example
\begin{equation*}
    \bm{u}=\bm{0},\quad p=\lambda(y^3 - y^2/2 + y - 7/12),
\end{equation*}
where $\lambda=10^7$.
Since $\bm{u}\in V_h$, we have $\norm{\bm{u}-J_h\bm{u}}_{L^2(\Omega)}=0$.
Therefore, one can expect that a pressure-robust method produces a discrete velocity that coincides with the exact velocity.
In this example, we also implemented the formulation derived in \cite{Zhao2019} for comparison.
Since \cite{Zhao2019} considered the Stokes equations with $k=0$, we extended the formulation proposed therein to arbitrary order polynomial and handled the nonlinear term as in the current formulation.
In Figure~\ref{fig:noflow_conv}, the convergence history of $L^2$-errors obtained with quasi-uniform polygonal meshes and quadratic polynomial spaces is depicted.
While the discrete velocity obtained from \cite{Zhao2019} converges, its magnitude is almost the same as the pressure error.
This shows that the velocity error depends on not only $(\bm{u}-J_h\bm{u})$ but also $(p-I_hp)$.
On the other hand, the discrete velocity produced by the proposed method is around $10^{-10}$ regardless of the mesh size.
Here, $10^{-10}$ is compatible with the machine precision multiplied by the condition number of the resulting matrix.
We also included the velocity magnitude profiles, in Figure~\ref{fig:noflow}, obtained from the two formulations when $h\approx 1/8$.
From those two observations, we can conclude that the proposed method indeed leads to a pressure-robust velocity approximation.

\subsection{Lid-driven cavity}
\begin{figure}
    \centering
    \includegraphics[width=0.32\textwidth]{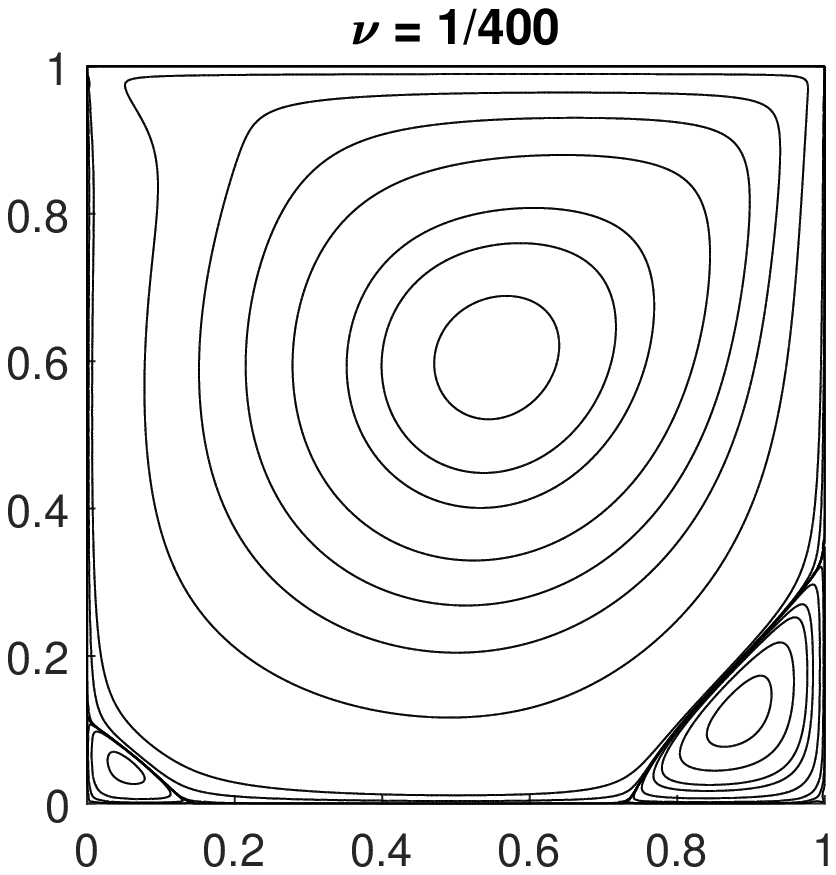}
    \includegraphics[width=0.32\textwidth]{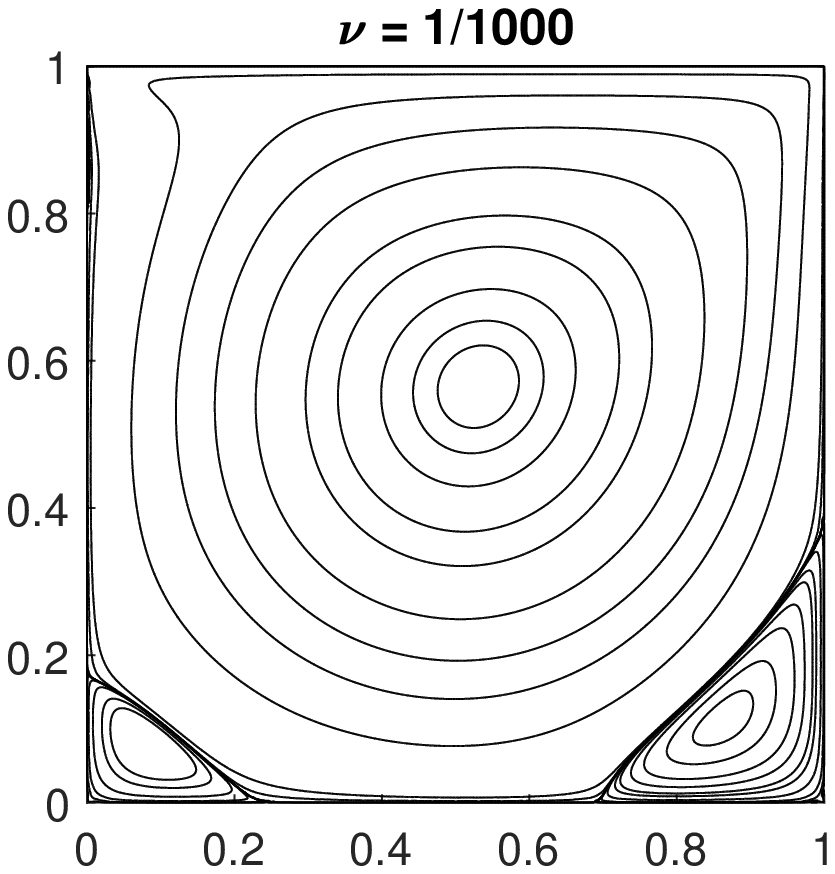}
    \includegraphics[width=0.32\textwidth]{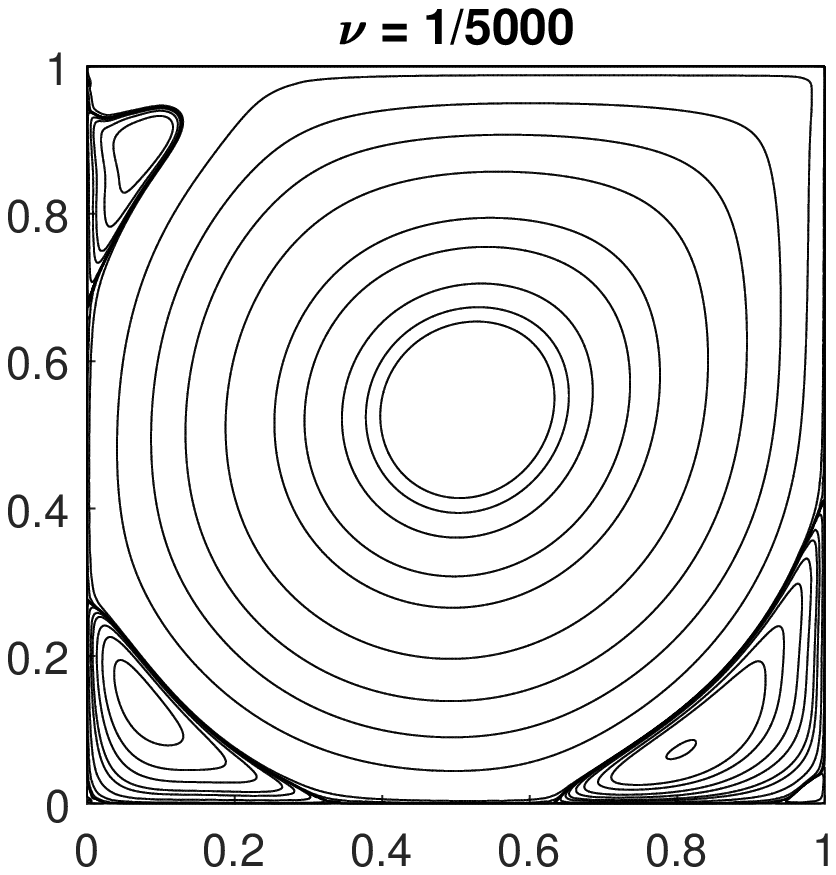}
    \caption{Streamline of lid-driven cavity with $h=1/32$ and $k=2$.}
    \label{fig:cavity}
\end{figure}
In the last example, we consider the famous lid-driven cavity problem to demonstrate the performance of the proposed method.
Consider a rectangular domain $(0,1)\times(0,1)$.
We choose $\bm{f}=\bm{0}$ and the Dirichlet boundary condition is given by
\begin{equation*}
    \bm{u}=\begin{cases}[1,0]^T&\text{on }(0,1)\times\{y=1\},\\\bm{0}&\text{otherwise.}\end{cases}
\end{equation*}
Piecewise quadratic polynomials on a uniform rectangular grid with $h=1/32$ is used.
In Figure~\ref{fig:cavity}, contour plots of streamfunctions for $\nu^{-1}=400,\;1000,\;5000$ are displayed.
Here, the contour levels are chosen as in \cite{ghia1982}.
We can observe that the streamlines are qualitatively similar to that of \cite{ghia1982}.

\section{Conclusion}\label{sec:con}
In this study, we developed a high-order polygonal staggered DG method for the Navier-Stokes equations.
A new finite element pair is introduced in the spirit of \cite{Zhao2020c}, which enables us to achieve pressure robustness.
It is worth mentioning that the proposed scheme yields a divergence free velocity approximation without any postprocessing which is a desirable feature in the applications of (Navier-)Stokes equations.
Another novel contribution lies in the design a new nonlinear convective term that earns non-negativity.
A full convergence analysis is carried out, showing that the error estimate for velocity is independent of pressure and of the viscosity.
Finally, several numerical experiments are carried out, and we can observe that our scheme is indeed pressure robust. The numerical experiments indicate that our method is a good candidate for practical applications in particular for problems with high Reynolds number.

\appendix
\section{Proof of Lemma~\ref{lem:Nh-uu-uhuh}}\label{sec:app}
In this appendix, we prove Lemma~\ref{lem:Nh-uu-uhuh}.
\begin{proof}
    By adding and subtracting $N_h(J_h\bm{u};J_h\bm{u},\bm{v}_h)$, we obtain
    \begin{equation}\label{eq:Nh_etas}
        \begin{aligned}
            N_h(\bm{u};\bm{u},\bm{v}_h)-N_h(\bm{u}_h;\bm{u}_h,\bm{v}_h)
            &=N_h(\bm{e}_*;\bm{u},\bm{v}_h)+N_h(J_h\bm{u};\bm{e}_*,\bm{v}_h)\\
            &\quad N_h(\bm{e}_h;\bm{u}_h,\bm{v}_h)+N_h(J_h\bm{u};\bm{e}_h,\bm{v}_h).
        \end{aligned}
    \end{equation}
    We will bound each term separately.
    First we have
    \begin{equation*}
        \begin{aligned}
            N_h(\bm{e}_*;\bm{u},\bm{v}_h)
            &=-(\bm{u}\otimes\bm{e}_*,\nabla\bm{v}_h)+\einner{\avg{\bm{e}_*\cdot\nn},\avg{\bm{u}}\cdot\jump{\bm{v}_h}}_{\mathcal{F}_h}.
        \end{aligned}
    \end{equation*}
    The first term on the right-hand side can be bounded using H\"older's inequality and the Sobolev inequality
    \begin{equation*}
        -(\bm{u}\otimes\bm{e}_*,\nabla\bm{v}_h)\leq C\norm{\bm{u}}_{H^1(\Omega)}\norm{\bm{e}_*}_{L^4(\mathcal{T}_h)}\norm{\nabla\bm{v}_h}_{L^2(\mathcal{T}_h)}.
    \end{equation*}
    Let us focus on the second term.
    For $e\in \mathcal{F}_h$, the trace inequality \eqref{eq:traceL4} implies that
    \begin{equation*}
        \begin{aligned}
            \einner{\avg{\bm{e}_*\cdot\nn},\avg{\bm{u}}\cdot\jump{\bm{v}_h}}_e
            &\leq \norm{\avg{\bm{e}_*\cdot\nn}}_{L^4(e)}\norm{\avg{\bm{u}}}_{L^4(e)}\norm{\jump{\bm{v}_h}}_{L^2(e)}\\
            &\leq C\left(\norm{\bm{u}}_{L^4(T)}
            +h_T^{1/2}\norm{\bm{u}}_{H^1(T)}\right)h_T^{1/4}\norm{\avg{\bm{e}_*\cdot\nn}}_{L^4(e)}h_T^{-1/2}\norm{\jump{\bm{v}_h}}_{L^2(e)},
        \end{aligned}
    \end{equation*}
    where $T$ is a neighboring triangle of $e$.
    By using shape regularity $h_T\leq \rho h_e$, summing over edges yields
    \begin{equation*}
        \begin{aligned}
            \einner{\avg{\bm{e}_*\cdot\nn},\avg{\bm{u}}\cdot\jump{\bm{v}_h}}_{\mathcal{F}_h}\leq C\norm{\bm{u}}_{L^4(\Omega)}\left(\sum_{e\in\mathcal{F}_h}h_e\norm{\avg{\bm{e}_*\cdot\nn}}_{L^4(e)}^4\right)^{1/4}\left(\sum_{e\in\mathcal{F}_h}h_e^{-1}\norm{\jump{\bm{v}_h}}_{L^2(e)}^2\right)^{1/2}.
        \end{aligned}
    \end{equation*}
    By combining these, we obtain
    \begin{equation*}
        N_h(\bm{e}_*;\bm{u},\bm{v}_h)\leq C\norm{\bm{u}}_{H^1(\Omega)}\norm{\bm{e}_*}_{0,4,h}\norm{\bm{v}_h}_h,
    \end{equation*}
    where $\norm{\cdot}_{0,4,h}$ is defined in \eqref{eq:discL4norm}.
    Similarly, we obtain
    \begin{equation*}
        N_h(J_h\bm{u};\bm{e}_*,\bm{v}_h)\leq C\norm{J_h\bm{u}}_h\norm{\bm{e}_*}_{0,4,h}\norm{\bm{v}_h}_h\leq C\norm{\bm{u}}_{H^1(\Omega)}\norm{\bm{e}_*}_{0,4,h}\norm{\bm{v}_h}.
    \end{equation*}
    It remains to estimate the last two terms on the right hand side of \eqref{eq:Nh_etas} . Observe that
    \begin{equation*}
        \begin{aligned}
            N_h(\bm{e}_h;\bm{u}_h,\bm{v}_h)
            &=-(\bm{u}_h\otimes\bm{e}_h,\nabla\bm{v}_h)+\einner{\avg{\bm{e}_h\cdot\nn},\avg{\bm{u}_h}\cdot\jump{\bm{v}_h}}_{\mathcal{F}_h}+\einner{|\avg{\bm{e}_h\cdot\nn}|,\jump{\bm{u}_h}\cdot\jump{\bm{v}_h}}_{\mathcal{F}_h}\\
            &\leq \norm{\bm{u}_h}_{L^4(\Omega)}\norm{\bm{e}_h}_{L^4(\Omega)}\norm{\nabla\bm{v}_h}_{L^2(\mathcal{T}_h)}\\
            &\quad+\sum_{e\in\mathcal{F}_h}\left(\norm{\avg{\bm{u}_h}}_{L^4(e)} + \norm{\jump{\bm{u}_h}}_{L^4(e)}\right)\norm{\avg{\bm{e}_h\cdot\nn}}_{L^4(e)}\norm{\jump{\bm{v}_h}}_{L^2(e)}.
        \end{aligned}
    \end{equation*}
    The discrete trace inequality \eqref{eq:discTrace} yields
    \begin{equation*}
        \begin{aligned}
        N_h(\bm{e}_h;\bm{u}_h,\bm{v}_h)
        &\leq \norm{\bm{u}_h}_{L^4(\Omega)}\norm{\bm{e}_h}_{L^4(\Omega)}\norm{\nabla\bm{v}_h}_{L^2(\mathcal{T}_h)}\\
        &\quad+C\norm{\bm{e}_h}_{L^4(\Omega)}\norm{\bm{u}_h}_{L^4(\Omega)}\left(\sum_{e\in\mathcal{F}_h}h_e^{-1}\norm{\jump{\bm{v}_h}}_{L^2(e)}^2\right)^{1/2}.
        \end{aligned}
    \end{equation*}
    By \eqref{eq:Lq}, we obtain
    \begin{equation*}
        N_h(\bm{e}_h;\bm{u}_h,\bm{v}_h)\leq C\norm{\bm{u}_h}_h\norm{\bm{e}_h}_h\norm{\bm{v}_h}_h.
    \end{equation*}
    Similarly, we obtain
    \begin{equation*}
        N_h(J_h\bm{u};\bm{e}_h,\bm{v}_h)\leq C\norm{J_h\bm{u}}_h\norm{\bm{e}_h}_h\norm{\bm{v}_h}_h\leq C\norm{\bm{u}}_{H^1(\Omega)}\norm{\bm{e}_h}_h\norm{\bm{v}_h}_h.
    \end{equation*}
    Combining the preceding estimates with Lemma~\ref{lem:bound_by_f0} and Theorem~\ref{thm:existence} concludes the proof.
\end{proof}

\section*{Acknowledgments}

The research of Eric Chung is partially supported by the Hong Kong RGC General Research Fund (Project numbers 14304719 and 14302018)
and the CUHK Faculty of Science Direct Grant 2020-21.
The research of Eun-Jae Park was supported by the National Research Foundation of Korea (NRF) grant funded by the Ministry of Science and ICT (NRF-2015R1A5A1009350 and NRF-2019R1A2C2090021).

\bibliographystyle{siamplain}
\bibliography{references}


\end{document}